\documentclass[11pt, reqno]{amsart}
\usepackage{latexsym,amsmath,amssymb, amsthm, mathscinet, mathtools}
\usepackage{cases, verbatim}
\usepackage[active]{srcltx}
\usepackage{hyperref}
\usepackage{xcolor}

\usepackage[margin=1.5in]{geometry}

\newtheorem{thm}{Theorem}
\newtheorem{lem}[thm]{Lemma}
\newtheorem{prop}[thm]{Proposition}

\theoremstyle{remark}
\newtheorem{rmk}[thm]{Remark}
\newtheorem{example}[thm]{Example}
\newtheorem{conj}[thm]{Conjecture}

\numberwithin{equation}{section}

\theoremstyle{definition}
\newtheorem{defi}[thm]{Definition}

\numberwithin{thm}{section} 
\makeatletter

\newcommand{\Rmnum}[1]{\expandafter\@slowromancap\romannumeral #1@}
\makeatother

\def\R{{\mathbb R}}

\def\H{{\mathbb H}}

\def\S{{\mathbf S}}

\newcommand{\vep}{\varepsilon}
\newcommand{\ol}{\overline}

\newcommand{\dive}{\operatorname{div}}
\newcommand{\tr}{\operatorname{tr}}

\newcommand{\bpm}{\begin{pmatrix}}
\newcommand{\epm}{\end{pmatrix}}
\newcommand{\la}{\left\langle}
\newcommand{\ra}{\right\rangle}

\newcommand{\beq}{\begin{equation}}
\newcommand{\eeq}{\end{equation}}
\newcommand{\spa}{\operatorname{span}}

\DeclareMathOperator*{\limsups}{limsup^\ast}

plus 6pt minus 12pt
\title[Horizontal quasiconvexity operator]{\protect{A second-order operator for horizontal quasiconvexity in the Heisenberg group \\ and application to convexity preserving for horizontal curvature flow}}

\author[A. Kijowski]{Antoni Kijowski}
\address[Antoni Kijowski]{Analysis on Metric Spaces Unit, Okinawa Institute of Science and Technology Graduate University, Okinawa 904-0495, Japan, {\tt antoni.kijowski@oist.jp}}
\author[Q. Liu]{Qing Liu}
\address[Qing Liu]{Geometric Partial Differential Equations Unit, Okinawa Institute of Science and Technology Graduate University, Okinawa 904-0495, Japan, {\tt qing.liu@oist.jp}}
\author[Y. Zhang]{Ye Zhang}
\address[Ye Zhang]{Analysis on Metric Spaces Unit, Okinawa Institute of Science and Technology Graduate University, Okinawa 904-0495, Japan, {\tt ye.zhang2@oist.jp}}

\author[X. Zhou]{Xiaodan Zhou}
\address[Xiaodan Zhou]{Analysis on Metric Spaces Unit, Okinawa Institute of Science and Technology Graduate University, Okinawa 904-0495, Japan, {\tt xiaodan.zhou@oist.jp}}

\begin{document}

\begin{abstract}
This paper is concerned with a PDE approach to horizontally quasiconvex (h-quasiconvex) functions in the Heisenberg group based on a nonlinear second order elliptic operator. We discuss sufficient conditions and necessary conditions for upper semicontinuous, h-quasiconvex functions in terms of the viscosity subsolution to the associated elliptic equation. Since the notion of h-quasiconvexity is equivalent to the horizontal convexity (h-convexity) of the function's sublevel sets, we further adopt these conditions to study the h-convexity preserving property for horizontal curvature flow in the Heisenberg group. Under the comparison principle, we show that the curvature flow starting from a star-shaped h-convex set preserves the h-convexity during the evolution. 
\end{abstract}


\maketitle

\section{Introduction}\label{sec:intro}

\subsection{Background}
This paper is devoted to studying a nonlinear elliptic operator for horizontally quasiconvex (h-quasiconvex) functions in the first Heisenberg group $\H$ and applying the second order characterization to investigate the convexity preserving property of the horizontal curvature flow. This paper is closely related with our previous work \cite{KLZ} on a first order characterization for h-quasiconvex functions based on a nonlinear nonlocal Hamilton-Jacobi operator. Our focus in this paper is different. Inspired by the work \cite{BGJ13} of Barron, Goebel and Jensen, we look into a second order operator instead, which leads to a new application to geometric properties of the curvature flow equation in the Heisenberg group. 

Let us briefly go over several basic notions about the Heisenberg group. See \cite{CDPT} for a detailed introduction. The Heisenberg group $\mathbb{H}$ is $\mathbb{R}^{3}$ endowed with the non-commutative group multiplication 
\[
(x_p, y_p, z_p)\cdot (x_q, y_q, z_q)=\left(x_p+x_q, y_p+y_q, z_p+z_q+\frac{1}{2}(x_py_q-x_qy_p)\right),
\]
for all $p=(x_p, y_p, z_p)$ and $q=(x_q, y_q, z_q)$ in $\H$. For any smooth function $f$,
we define the horizontal gradient $\nabla_H f$ to be 
\[
\nabla_H f=(X_1 f, X_2 f),
\]
where $X_1 f$ and $X_2 f$ denote the horizontal derivatives of $f$ determined by the left-invariant vector fields
\[
X_1=\frac{\partial}{\partial x}-\frac{y}{2}\frac{\partial}{\partial z}, \quad  X_2=\frac{\partial}{\partial y}+\frac{x}{2}\frac{\partial}{\partial z}.
\]
Similarly, the (symmetrized) horizontal Hessian of $f$ is defined by 
\[
(\nabla_H^2 f)^\star=
\begin{pmatrix}
 X_1^2 f & (X_1X_2 +X_2X_1 f)/2\\
 (X_1X_2 +X_2X_1 f)/2 & X_2^2 f
\end{pmatrix}.
\]
We denote by $\dive_H f$ the horizontal divergence for a smooth vector valued function $f=(f_1, f_2): \H\to \R^2$, i.e., $\dive_H f=X_1 f_1+X_2 f_2$.
%define the hori The differential structure of $\mathbb{H}$ is determined by 
  %One may easily verify the commutation relation $X_3=[X_1, X_2]=
  %X_{1}X_{2}- X_{2}X_{1}$. Let
Let $\H_0$ denote the horizontal plane through the group identity $0$, that is, 
\[
\mathbb{H}_0=\{h\in \mathbb{H}: h=(x, y, 0) \ \text{ for $x, y\in \mathbb{R}$}\}.
\]
For any $p\in \H$, the set
\[
\H_p=\{p\cdot h: \ h\in \H_0\}
\]
is called the horizontal plane through $p$. It is clear that $\H_p=\spa\{X_1(p), X_2(p)\}$ for every $p\in \H$. Also, a line segment $[p, q]$ in $\H$ is said to be horizontal if $p\in \H_q$. 

We are interested in the notion of so-called h-convex sets in $\H$, which is proposed by \cite{DGN1} as a possible extension of Euclidean convex sets to the sub-Riemannian setting. A set $E\subset \H$ is said to be h-convex if the horizontal segment connecting any two points in $E$ lies in $E$. The authors of \cite{DGN1} name such $E$ a weakly h-convex set but for simplicity hereafter we refer to it as h-convex set. Consult \cite{Rithesis, CCP1, ArCaMo} etc. for various properties about this notion. Other notions of set convexity in the Heisenberg group such as geodesic convexity and strong h-convexity and discussions on their relations can be found in \cite{LuMaSt, MoR, DGN1, Rithesis, CCP1}

The h-convexity is obviously much weaker than the notion of convexity in the Euclidean space. As pointed out in \cite{CCP2}, h-convex sets do not even need to be connected. In fact, the union of two distinct points on the $z$-axis is an h-convex set, as there are no horizontal segments connecting the points. Such a special feature causes much difficulty in studying h-convex sets directly. We turn to a more analytic approach, examining instead the so-called h-quasiconvex functions in the Heisenberg group, which are defined to be the functions whose sublevel sets are h-convex; see Definition \ref{hqc}. It is a natural counterpart of the quasiconvex functions in the Euclidean space studied in \cite{SuYa, CCP2}, but again it is a much weaker notion than the Euclidean quasiconvexity, which requires all sublevel sets of the function to be convex in the Euclidean space. The introduction of h-quasiconvexity enables us to incorporate PDE methods into our analysis of h-convex sets. We expect that this formulation will bring us new insights, as in the Euclidean case it successfully provides PDE-based characterizations for general quasiconvex functions \cite{BGJ1, BGJ13}.

Our goal is to extend the Euclidean approaches introduced by Barron, Goebel and Jensen \cite{BGJ1, BGJ13} to the Heisenberg group and explore their applications in convex analysis and PDE theory on sub-Riemannian manifolds. As mentioned previously, our prior work \cite{KLZ} gives a characterization for h-quasiconvex functions based on a sub-Riemannian analogue of a first order nonlocal operator in \cite{BGJ1} and applies this characterization to the contruction of h-quasiconvex envelope and h-convex hull. 
%showing that an upper semicontinuous (USC) function $f$ in an h-convex set $\Omega\subset \H$ is h-quasiconvex if and only if $f$ satisfies in the viscosity sense
%\[
%\sup\{\la \nabla_H f(p), (p^{-1}\cdot q)_h\ra: q\in \H_p\cap \Omega, \ f(q)< f(p)\}\leq 0, \quad \text{for $p\in \Omega$.}
%\]
%It is as a sub-Riemannian analogue of the characterization for quasiconvex functions in the Euclidean space given . Also, in the framework of viscosity solution theory, it improves the result in \cite{CCP2} by broadening the function class from $C^1(\Omega)$ to merely $USC(\Omega)$.
Aimed to facilitate broader applications, our current work attempts to develop a distinct PDE-based characterization inspired by \cite{BGJ13}. We will compare our results with Euclidean approach and discuss in detail an application to the horizontal curvature flow in the Heisenberg group. 

\subsection{Necessary and sufficient conditions for horizontal quasiconvexity}
The nonlinear operator for Euclidean quasiconvexity proposed in \cite{BGJ13} is of second order and takes the form 
\begin{equation}\label{sec operator}
L_{eucl}[f](x)=\min\{\la \nabla^2 f(x) \eta, \eta\ra: \eta\in \R^n, |\eta|=1, \la \nabla f(x), \eta\ra=0\}
\end{equation}
for any $x\in \Omega$. 
It is not difficult to see at least formally that the sign of $L_{eucl}[f]$ is closely linked with the quasiconvexity of $f$. Note that when $f$ is of class $C^2(\Omega)$ and $\nabla f(x)\neq 0$, the quantity {$L_{eucl}[f](x)/|\nabla f(x)|$} %$L_0(\nabla f(p), \nabla^2 f(p))/|\nabla f(p)|)$ 
represents the least principal curvature of the level surface of $f$ at $x$. In fact, for $f\in USC(\Omega)$ in a convex domain $\Omega$, the following results are obtained in \cite{BGJ13}:

\begin{enumerate}
    \item[(a)] If $f$ is quasiconvex in $\Omega$, then $L_{eucl}[f]\geq 0$ in $\Omega$ holds in the viscosity sense;
    \item[(b)] If $L_{eucl}[f]> 0$ in $\Omega$ holds in the viscosity sense, then $f$ is quasiconvex in $\Omega$;
    \item[(c)] If $L_{eucl}[f]\geq 0$ in $\Omega$ holds in the viscosity sense and and $f$ does not attain local maxima in $\Omega$, then  $f$ is quasiconvex in $\Omega$.
\end{enumerate}
A more detailed review about these results including the definition of viscosity subsolutions is given in Section~\ref{sec:eucl}. 

It is an intriguing question whether there is a second order operator that has similar properties in the Heisenberg group $\H$. A natural substitute of $L_{eucl}$ in $\H$ is given by 
\begin{equation}\label{sec operator heis}
L[f](p)=\min\{\la (\nabla_H^2 f)^\star(p) \eta, \eta\ra: \eta\in \R^2, |\eta|=1, \la \nabla_H f(p), \eta\ra=0\} %\quad \text{$p\in \Omega$.}
\end{equation}
for $f\in C^2(\Omega)$ and $p\in \Omega$. When $\nabla_H f(p)\neq 0$, one can write $L[f](p)$ as
\begin{equation}\label{sec operator heis2}
\begin{aligned}
L[f](p)={1\over {|\nabla_H f(p) |^2}}\la (\nabla_H^2 f)^\star(p)\nabla_H f(p)^\perp, \nabla_H f(p)^\perp\ra=|\nabla_H f|\dive_H\left({\nabla_H f\over |\nabla_H f|}\right)(p). 
\end{aligned}
\end{equation}
Since the term $\dive_H\left({\nabla_H f/ |\nabla_H f|}\right)$ stands for the horizontal curvature of level sets of $f$, we expect that the h-quasiconvexity of $f$ can be determined by the sign of $L[f]$. Indeed, we obtain the following analogue of the results (a)(b) above in the Euclidean case. 

\begin{thm}[Characterization of H-quasiconvex functions]\label{main thm}
Let $\Omega$ be an h-convex open set in $\H$. Let $f\in USC(\Omega)$. Then $L[f]\geq 0$ in $\Omega$ holds in the viscosity sense if $f$ is h-quasiconvex in $\Omega$. Moreover, $f$ is h-quasiconvex in $\Omega$ if $L[f]>0$ in $\Omega$ holds in the viscosity sense. 
\end{thm}

We will prove the sufficient condition and necessary condition for h-quasiconvexity separately in Section \ref{sec:necessary} and Section \ref{sec:sufficient}. It is however not clear to us whether the sub-Riemannian analogue of the statement (c) above holds. We do not know whether the viscosity inequality $L[f]\geq 0$ in $\Omega$ together with the nonexistence of local maxima of $f$ in $\Omega$ is sufficient to imply the h-quasiconvexity of $f$ in $\Omega$. 

We emphasize that in our sub-Riemannian case one needs to handle these viscosity inequalities carefully due to the discontinuity of the operator $L[f]$ at a characteristic point $p$ where $\nabla_H f(p)=0$ holds. Inspired by the standard theory of viscosity solutions, in addition to the original operator $L$, we also consider two variants of $L$ given by 
\begin{equation}\label{envelope2}
\ol{L[f]}(p)=\limsup_{q\to p}L[f](q),
\end{equation}
\begin{equation}\label{envelope1}
L^\ast[f](p)=\limsup_{\substack{\xi\to \nabla_H f(p)\\ X\to (\nabla_H^2 f)^\star(p)}} \min\{\la X \eta, \eta \ra: \eta\in \R^2, |\eta|=1, \la \xi, \eta\ra=0\}    
\end{equation}
for any $f\in C^2(\Omega)$ and $p\in \Omega$. It is not difficult to see that for such $f$ and $p$, 
\begin{equation}\label{compare}
L[f](p)\le \ol{L[f]}(p)\le L^\ast[f](p).
\end{equation}
%for any $f\in C^2(\Omega)$ and $p\in \Omega$. 

Note that taking such upper semicontinuous envelopes for the Euclidean operator $L_{eucl}$ does not make any difference in the results (a)(b)(c) for quasiconvexity in the Euclidean space; see Remark \ref{rmk weak}. However, applying weaker viscosity inequalities with the envelopes in \eqref{envelope2} or \eqref{envelope1} in $\H$ results in quite different scenarios. In Example \ref{positive}, we find that the function 
\[
f(x,y,z)= x^2 + \left(z +{ xy\over 2}\right)^2
\]
satisfies $L^\ast[f]>0$ and $\ol{L[f]}\geq 0$ in $\H$ but is not h-quasiconvex in $\H$. Since $f$ does not achieve any local maximum in $\H$, this example also shows that the sufficient condition like that in (c) fails to imply h-quasiconvexity of $f$ if we adopt the inequality  $L^\ast[f]\geq 0$ or $\ol{L[f]}\geq 0$ instead of $L[f]\geq 0$. Such a discrepancy in the properties of these operators, caused by the singularity at the characteristic points, constitutes a significant difference between the Euclidean space and Heisenberg group. 

In addition to the results described above, for our further applications we also introduce a stronger notion of h-quasiconvexity, which we call uniform h-quasiconvexity in this paper. A uniformly h-quasiconvex function is related to the viscosity inequality $L[f]\geq c$ for a constant $c>0$. More precise definition and properties of uniformly h-quasiconvex functions will be elaborated in Section \ref{sec:uniform}.

\subsection{Application to horizontal curvature flow}
As an application of our analysis on h-quasiconvex functions, we study the h-convexity preserving property of the motion by horizontal curvature. By using the level set formulation, we can write the equation as follows: 
\begin{numcases}{}
u_t-|\nabla_H u|\dive_H(\nabla_H u/|\nabla_H u|)=0 \quad &\text{in $\H\times (0, \infty)$,} \label{mcf}\\
u(\cdot, 0)=u_0 \quad &\text{in $\H$,} \label{initial}
\end{numcases}
where $u_0\in C(\H)$ is a given initial value. Note that for any smooth function $u$ and  $(p, t)\in \H\times (0, \infty)$ with $\nabla_H u(p, t)\neq 0$, $u_t(p, t)/|\nabla u(p, t)|$ and $\dive_H(\nabla_H u(p, t)/|\nabla_H u(p, t)|)$ respectively denote the normal velocity and curvature of the level surface $\{u(\cdot, t)=c\}$  at $p$ with the level $c=u(p, t)$ \cite{CDPT, DGN3}. In general, we cannot expect that the solution $u$ is smooth due to the degeneracy and nonlinearity of the parabolic operator. One may study the Cauchy problem in the framework of viscosity solutions.  We refer to \cite{CC1, DDR, FLM1, Oc, BCi, DDG} concerning well-posedness results for this problem. See other related discussions on this topic in \cite{CCS, CiLeSa, DSo}.  

We remark that the general uniqueness of viscosity solutions of \eqref{mcf}\eqref{initial} still remains an open question. In this paper, we thus assume that the comparison principle below holds.
\begin{enumerate}
\item[(CP)] Let $C\in \R$ and $K$ be a compact set of $\H$. Let $u\in USC(\H\times [0, \infty))$ and $v\in LSC(\H\times [0, \infty))$ be respectively a subsolution and a supersolution of \eqref{mcf} satisfying $u, v\leq C$ in $\H\times [0, \infty)$ and $u=v=C$ in $(\H\setminus K)\times [0, \infty)$. If $u\leq v$ in $\H\times \{0\}$, then $u\leq v$ in $\H\times [0, \infty)$.
\end{enumerate}
Such a comparison principle is established in \cite{FLM1} under the rotational symmetry (with respect to $z$-axis) of the sub- and supersolutions. More recently, the uniqueness is addressed in \cite{BCi} for solutions that are built by the vanishing viscosity method. However, it is still not clear whether (CP) holds for general viscosity sub- and supersolutions.  

Our focus is to discuss the following h-convexity preserving property for the motion by curvature in the Heisenberg group. For a given bounded open h-convex set $E_0 \subset \H$, let $u_0\in C(\H)$ be a function satisfying 
\begin{equation}\label{initial zero}
E_0=\{p\in \H: u_0(p)< 0\},
\end{equation}
 then $E_t=\{u(\cdot, t)< 0\}$ is h-convex for all $t\geq 0$. Convexity preserving property is well known for mean curvature flow in the Euclidean space \cite{Hui, GH}. We refer also to \cite{GGIS, ALL, LSZ} etc. for different approaches to convexity of viscosity solutions of the level set equation. 
 
In the Heisenberg group, although a similar preserving property for h-convexity is also expected to hold, it is not clear how to adapt the PDE methods in \cite{GGIS, ALL} in our sub-Riemannian case. The method in \cite{ALL} is generalized in the Heisenberg group \cite{LMZ, LZ2} for a class of nonlinear parabolic and elliptic equations under certain symmetry on the solutions, but the elliptic operator is required to be concave in the horizontal gradient, which does not fit the current case of horizontal curvature flow equation. 

On the other hand, the h-quasiconvexity preserving property for \eqref{mcf} can be heuristically observed from our preceding results. As clarified in \eqref{sec operator heis2}, the term involving the curvature agrees with our h-quasiconvexity operator $L$. Assuming that $u_0$ is h-quasiconvex in $\H$, we can apply Theorem \ref{main thm} to obtain $L[u_0]\geq 0$ in $\H$ in the viscosity sense. Formally, this implies that $u_t\geq 0$ at $t=0$. It then follows from a comparison argument that $u_t\geq 0$ for all $t\geq 0$, which in turn yields $L[u(\cdot, t)]\geq 0$ for all $t\geq 0$. Although our sufficient condition for h-quasiconvexity of $u(\cdot, t)$ as in Theorem \ref{main thm} actually requires a strict inequality $L[u(\cdot, t)]> 0$ in $\H$, it is already quite close to our goal of proving the h-convexity of $E_t$. 

 In order to close the gap at the final step above, we instead adopt the notion of uniform h-quasiconvexity so that for any $c>0$ the same formal argument enables us to obtain $L[u(\cdot, t)]\geq c$ for all $t\geq 0$ from initial value satisfying $L[u_0]\geq c$ in $\H$. We then can use Theorem \ref{main thm} to conclude the h-quasiconvexity of $u(\cdot, t)$. 
%Let us state our result on the preservation of h-quasiconvexity is as follows. 
\begin{thm}[H-quasiconvexity preserving property]\label{cor uni-qc2}
Suppose that (CP) holds. Let $C\in \R$. Assume that $u_0\in C(\H)$ satisfies $u_0\leq C$ in $\H$ and $u_0=C$ outside a compact set of $\H$. Assume further that there exists $\hat{u}_0\in C(\H)$ uniformly h-quasiconvex in $\H$ satisfying 
\begin{equation}\label{eq growth}
\hat{u}_0(p)\leq L(|p|_G^4+1) \quad \text{in $\H$ for some $L>0$,}
\end{equation}
where $|\cdot|_G$ is the Kor\'{a}nyi gauge defined as in \eqref{gauge} below, and 
\begin{equation}\label{truncation}
u_0=\min\{\hat{u}_0, C\} \quad \text{in $\H$}.
\end{equation}
 Let $u$ be the unique solution of \eqref{mcf}\eqref{initial}. Then,  $u(\cdot, t)$ is h-quasiconvex in $\H$ for all $t\geq 0$.
\end{thm}
In this result we take $u_0$ to be a truncation of a uniformly h-quasiconvex function $\hat{u}_0$ satisfying the growth condition \eqref{eq growth}. For more general h-quasiconvex initial data, we need to approximate them by truncated uniformly h-quasiconvex functions that satisfy the assumptions on $u_0$ in Theorem \ref{cor uni-qc2}. A more precise description about such generalization is presented in Theorem \ref{thm qc}. 

Our rigorous proof of Theorem \ref{cor uni-qc2} employs a game-based approximation for the horizontal curvature flow established in \cite{FLM1}, which assists us in tracking the spatial h-quasiconvexity of the approximate solution throughout the evolution. This game-theoretic interpretation is a sub-Riemannian generalization of that proposed by Kohn and Serfaty \cite{KS1} in the Euclidean space.

As our goal is to study the h-convexity preserving of set evolution of horizontal curvature flow starting from a given h-convex set $E_0$, when applying Theorem \ref{thm qc} we face another important question about the existence of h-quasiconvex function $u_0\in C(\H)$ that satisfies \eqref{initial zero} as well as the assumptions in Theorem \ref{thm qc}. In the Euclidean space, one can resolve this issue by simply taking $u_0$ to be the signed Euclidean distance to $E_0$, which serves as a quasiconvex defining function for $E_0$. The situation in the Heisenberg group is different. 
The distance function to an h-convex set $E$ is not necessarily h-quasiconvex. Even when $E=\{0\}$, it is well-known that if we use the Carnot-Carath\'eodory (CC) metric in $\H$, the sublevel sets of the CC-distance $d_{CC}(\cdot, 0)$, which are CC-balls, are not h-convex in $\H$. 

In Section \ref{sec:initial}, we give an affirmative answer to the existence problem of h-quasiconvex $u_0\in C(\H)$ under an additional star-shaped assumption on $E_0$ and a uniform h-convexity condition on $E_0$. An overview about star-shaped sets in Carnot groups is given in \cite{DrFi}. Our construction is based on a Minkowski-type functional for $E_0$. % We construct a uniformly h-quasiconvex function $\hat{u}_0\in C(\H)$ based on the Minkowski-type functional for $E_0$. The desired function $u_0$ can be obtained by truncating $\hat{u}_0$ with a constant $C>0$ large. 
One can further utilize the approximation of $u_0$ as introduced in Theorem \ref{thm qc} to handle more general h-convex initial sets $E_0$. We discuss a special case of rotationally symmetric surface evolution in Proposition \ref{prop starshape sym} and Proposition \ref{prop initial approx}, where more specific assumptions on the initial value are provided. Concrete examples are also presented at the end of Section \ref{sec:initial}. 

\subsection{Notations}
We conclude the introduction with more notations that will be used in the work. We will write $|\cdot |_G$ to denote the Kor\'{a}nyi gauge, i.e., for $p=(x, y, z)\in \H$
\begin{equation}\label{gauge}
|p|_G=\left((x^2+y^2)^2+16z^2\right)^{\frac{1}{4}}.
\end{equation}
The Kor\'{a}nyi gauge induces a left invariant metric $d_H$ on $\H$ with
\[
d_H(p, q)=|p^{-1}\cdot q|_G\quad p, q\in \H.
\]
%We also use the right invariant metric $\tilde{d}_H$,  defined by 
%\beq\label{right metric}
%\tilde{d}_H(p, q)=|p\cdot q^{-1}|_G, \quad p, q\in \H.
%\eeq

%The associated distances between a point $p\in \H$ and a set $E\subset \H$ are respectively denoted by $d_H(p, E)$ and $\tilde{d}_H(p, E)$. 
%For two sets $D, E\subset \H$, we write $d_H(D, E)$ and $\tilde{d}_H(D, E)$ to denote respectively the Hausdorff distances between $D$ and $E$ with respect to the metrics $d_H$ and $\tilde{d}_H$, i.e, for $d=d_H$ or $d=\tilde{d}_H$,
%\[d(D, E)=\max\left\{\sup_{p\in D}d(p, E),\ \sup_{p\in E} d(p, D)\right\}.\]

We denote by $B_r(p)$ the open gauge ball in $\H$ centered at $p\in \H$ with radius $r>0$; namely,  
\[
B_r(p)=\{q\in \H: |p^{-1}\cdot q|_G< r\}.
\]
%while $\tilde{B}_r(p)$ represents the corresponding right-invariant metric ball. 
Let $\delta_\lambda$ denote the non-isotropic dilation in $\H$ with $\lambda\geq 0$, that is, $\delta_\lambda(p)=(\lambda x, \lambda y, \lambda^2 z)$ for $p=(x, y, z)\in \H$. 
We write $\delta_\lambda(E)$ to denote the dilation of a given set $E\subset \H$, that is, 
\[
\delta_\lambda(E)=\{\delta_\lambda(p): p\in E\}.
\]

\subsection*{Acknowledgments}
The work of QL was supported by JSPS Grant-in-Aid for Scientific 
Research (No.~19K03574, No.~22K03396).  
The work of XZ was supported by JSPS Grant-in-Aid for Research Activity Start-up (No.~20K22315) and JSPS Grant-in-Aid for Early-Career Scientists (No.~22K13947).

\section{Second order quasiconvexity operator in the Euclidean space}\label{sec:eucl}

%\subsection{The Euclidean quasiconvexity}

In order for our comparison with the sub-Riemannian setting, let us include a review of the Euclidean results \cite{BGJ13} on the connection between quasiconvex functions $f$ and the sign of $L_0(\nabla f, \nabla^2 f)$, where $L_0: \mathbb{R}^n\times \S^n\to \mathbb{R}$ is defined by
\begin{equation}\label{basic operator}
L_0(\xi, X)=\min\{\langle X\eta, \eta\rangle: \eta\in \mathbb{R}^n, |\eta|=1, \langle\xi,\eta\rangle=0\}.
\end{equation}
Here $\S^n$ denotes the set of $n\times n$ symmetric matrices. Recall that a function $f$ on a convex set $\Omega\subset \R^n$ is quasiconvex if
\begin{equation}
f(\lambda x + (1 - \lambda) y) \le \max\{f(x), f(y)\}, \quad \forall \, x, y \in \Omega, \, 0 < \lambda < 1,
\end{equation}
which is equivalent to the requirement that all sublevel sets of $f$ are convex.

%Let $\Omega$ be a convex open set in $\R^n$. 
For a given function $f:\Omega\to \mathbb{R}$, let $L_{eucl}[f]=L_0(\nabla f, \nabla^2 f)$, i.e., for $x\in \Omega$, $L_{eucl}[f](x)$ is given as in \eqref{sec operator}.
 % This operator is also related to  differential geometry, generalized convexity and  stochastic optimal control and games; 

Barron, Goebel and Jensen use this operator to establish necessary conditions and sufficient conditions for quasiconvex functions \cite{BGJ13}. %We recall their definitions and results below and give some remarks. 
%\begin{defi}\label{sub}
%A locally bounded function $f\in USC({\Omega})$ is said to be a viscosity subsolution of $L_{eucl}[f]\ge 0$ (resp., $L_{eucl}[f]> 0$) in $\Omega$ if whenever $f-\varphi$ achieves a strict local maximum at $p_0\in \Omega$ for a smooth function $\varphi: \Omega\to \mathbb{R}$, we have
%\[
%-L_{eucl}[\varphi](p_0)\le 0 \quad (\text{resp.,} \ \ -L_{eucl}[\varphi](p_0)<0).
%\]
%\end{defi}

%\begin{rmk}
%In spite of the name ``viscosity subsolutions'', such a notion is actually stronger than the standard viscosity subsolutions. It is well known that, when defining subsolutions for discontinuous equations, for stability purpose one usually uses the lower semicontinuous envelope of the operator in the viscosity inequality. In the current case, we should take

%\end{rmk}

%Let us give the definition of viscosity subolutions is given bel

%Let $L_0^*$ be the upper semicontinuous envelope of the operator; in other words,

\begin{defi}[Subsolutions associated to quasiconvexity operator]\label{sub}
A locally bounded function $f\in USC({\Omega})$ is said to be a viscosity subsolution of $L_{eucl}[f]\ge 0$ (resp., $L_{eucl}[f]> 0$) in $\Omega$, if whenever $f-\varphi$ achieves a strict local maximum at $\hat{x}\in \Omega$ for a smooth function $\varphi:\Omega\to \mathbb{R}$, we have
\begin{equation}\label{eq sub}
-L_{eucl}[\varphi](\hat{x})\le 0 \quad (\text{resp., } \ \  -L_{eucl}[\varphi](\hat{x})<0).
\end{equation}
\end{defi}

The following results, Theorems \ref{necessary}--\ref{localmax}, are taken from \cite{BGJ13}. 
%It turns out that being a viscosity subsolution of $L_{eucl}[f]\ge 0$ is a necessary condition of the Euclidean quasiconvexity.
\begin{thm}[{\cite[Theorem 2.6]{BGJ13}}]\label{necessary}
Let $\Omega$ be a convex open set in $\R^n$. If $f\in USC(\Omega)$ is quasiconvex in $\Omega$, then $f$ is a viscosity subsolution of $L_{eucl}[f]\ge 0$ in $\Omega$. 
\end{thm}

%It follows immediately that $f$ is also a viscosity subsolution of $L_{eucl}^\ast[f]\ge 0$ in $\Omega$. 

As pointed out in \cite[Example 1.1]{BGJ13}, in general the viscosity inequality $L_{eucl}[f]\ge 0$ is only a necessary condition and does not imply quasiconvexity of $f$. For example, one can easily verify that $f(x)=-x^4$ for $x\in \mathbb{R}$ is not quasiconvex (but quasiconcave) and satisfies $L_{eucl}[f]\ge 0$ in $\R$. The following result gives a sufficient condition for quasiconvexity in $\mathbb{R}^n$. 
\begin{thm}[{\cite[Theorem 2.7]{BGJ13}}]\label{strict}
Let $\Omega$ be a convex open set in $\R^n$. If $f\in USC(\Omega)$ is a viscosity subsolution of $L_{eucl}[f]>0$ in $\Omega$, then $f$ is quasiconvex in $\Omega$. 
\end{thm}

Another sufficient condition in \cite{BGJ13} with the weaker inequality $L_{eucl}[f]\ge 0$ is as below. 

\begin{thm}[{\cite[Theorem 2.8]{BGJ13}}]\label{localmax}
Let $\Omega$ be a convex open set in $\R^n$. If $f\in USC(\Omega)$ is a viscosity subsolution of $L_{eucl}[f]\ge0$ in $\Omega$ and $f$ does not attain a local maximum, then $f$ is quasiconvex in $\Omega$. 
\end{thm}

\begin{rmk}[Characterization with upper semicontinuous envelop]\label{rmk weak}
We would like to mention that %in the original proofs of these results, the authors of 
these results in \cite{BGJ13} still hold even if one weakens the definition of viscosity subsolutions of $L_{eucl}[f]\geq 0$ by replacing the inequality $-L_{eucl}[\varphi](\hat{x})\le 0$ at the maximizer $\hat{x}$ of $f-\varphi$ in $\Omega$ by
\begin{equation}\label{eq sub-weak}
-L_{eucl}^\ast[\varphi](\hat{x})\leq 0,
\end{equation}
where
$$L_{eucl}^\ast[f](x)=L_0^*(\nabla f(x), \nabla^2 f(x)).$$
Here $L_0^\ast$ denotes the upper semicontinuous envelope of the elliptic operator $L$, that is, 
\[
L_0^*(\xi, X)=\limsup_{\zeta\to \xi, Y\to X} L_0(\zeta, Y)=
\begin{cases}
L_0(\xi, X) & \text{if $\xi\neq 0$,}\\
\limsup_{\zeta\to 0, Y\to X} L_0(\zeta, Y)& \text{if $\xi= 0$}.
\end{cases}
\]
 It is clear that $L_{eucl}[f]\le L_{eucl}^\ast[f]$ in $\Omega$ for any $f\in USC(\Omega)$. 
It is a standard treatment in the viscosity solution theory to define solutions of discontinuous equations by adopting such semicontinuous envelopes of the operators; consult for example \cite{CIL, GBook}. 
Below let us give more details on this observation through the proof of Theorem \ref{localmax} (\cite[Theorem 2.7]{BGJ13}). Proofs for other Euclidean results with the relaxed definition of solutions are omitted here.
\end{rmk}

\begin{thm}[An improved sufficient condition for quasiconvexity]\label{localmax2}
Let $\Omega$ be a convex open set in $\R^n$. Suppose that $f\in USC(\Omega)$ satisfies $L_{eucl}^\ast[f]\geq 0$ in $\Omega$ in the sense that \eqref{eq sub-weak} holds whenever $f-\varphi$ attains a strict local maximum at $\hat{x}\in \Omega$ for a smooth function $\varphi: \Omega\to \R$. If $f$ does not attain a local maximum, then $f$ is quasiconvex in $\Omega$. 
\end{thm}

\begin{proof}%[Proof of Theorem \ref{localmax2} using the relaxed subsolution property \eqref{eq sub-weak}]
Supposing by contradiction that $f$  is not quasiconvex, by an affine change of variables and upper semicontinuity, we can 
assume there exist $w=(w_1, 0, \cdots, 0)$ with $w_1 \in (-1,1)$ such that $f(w) > f(\pm 1, [-1,1], \ldots, [-1,1])$ and 
\[
K:=[-1,1] \times [-2,2] \times \cdots \times [-2,2] \subset \Omega.
\]
Following the argument there, for sufficiently large $m$, $f-\varphi_m$ achieves a local strict maximum at an interior point $\hat{x}=(\hat{x}_1, \hat{x}_2, \ldots, \hat{x}_n)$ in $K$, where $\varphi_m$ is taken to be $$\varphi_m(x)=\frac{1}{m}(2-x_1^2)(x_2^m+\cdots+x_n^m)$$
for $x=(x_1, x_2, \ldots, x_n)\in \Omega$.
Also, we have $\hat{x}_2^2+\ldots +\hat{x}_n^2\neq 0$. In order to reach a contradiction, it remains to find $\eta$ appropriately in the definition \eqref{basic operator} of $L_0$ to get $L_{eucl}[\varphi_m](\hat{x})<0$. We choose $\eta$ to be
\[
\eta:= r_m(\hat{x}) (m(2 - \hat{x}_1^2), 2 \hat{x}_1 \hat{x}_2, \cdots, 2 \hat{x}_1 \hat{x}_n) \ne 0,
\]
where $r_m(\hat{x}) > 0$ is a normalizing constant so that $|\eta| = 1$. (Our choice of $\eta$ is slightly different from that in the original proof in \cite{BGJ13}, which does not seem to imply the desired inequality below.) By direct computations, we have $\la \nabla \varphi_m(\hat{x}), \eta\ra=0$ and
\begin{align*}
\langle \nabla^2 \varphi_m(\hat{x}) \eta, \eta\rangle = r_m(\hat{x})^2 (|\hat{x}|^2-\hat{x}_1^2)  (2 - \hat{x}_1^2)\left(-2  m (2 - \hat{x}_1^2) - 8 m \hat{x}_1^2 + 4\hat{x}_1^2 (m - 1)\right) < 0.
\end{align*}
%which implies $L_{eucl}[\varphi](\hat{x})<0$. 
Noticing that $\nabla \varphi_m(\hat{x}) \ne 0$, we are led to $L_{eucl}^*[\varphi_m](\hat{x})=L_{eucl}[\varphi_m](\hat{x})<0$.  
\end{proof}

As pointed out in \cite[Example 2.9]{BGJ13}, the condition that $f$ has no local maxima cannot be dropped. Note that the function {$f(x) = - (x^4 - 1)^4$} %$f(x)=(x^2-1)^2$ 
($x\in \R$) satisfies $L_{eucl}[f]\geq 0$ in $\R$ but $f$ is not quasiconvex in $\R$ and has a strict local maximum at $x=1$. 

%An analogous result also holds in the Heisenberg group, but a crucial difference in the sub-Riemannian case is that one cannot weaken the definition of subsolutions like \eqref{eq sub-weak}. See Theorem~\ref{strict heis} and Example~\ref{positive} below. 

%One can examine the proof \cite[Theorem 2.8]{BGJ13} to find that the condition $L_{eucl}[f]\geq 0$ can also be weakened to just $L_{eucl}^\ast[f]\geq 0$. 

%\begin{thm}\label{localmax improved}
%Let $\Omega$ be convex. If $f\in USC(\Omega)$ is a viscosity subsolution of $L_{eucl}^\ast[f]\ge0$ in $\Omega$ and $f$ does not attain a local maximum, then $f$ is quasiconvex in $\Omega$. 
%\end{thm}

\section{Horizontal quasiconvexity in the Heisenberg group}

\subsection{Second order h-quasiconvexity operator}
We next focus on h-quasiconvex functions in the Heisenberg group by generalizing the Euclidean operator $L_{eucl}$ and defining a sub-Riemannian version $L$ by \eqref{sec operator heis} for $f\in C^2(\Omega)$ and $p\in \Omega$, where $\Omega\subset \H$ is an h-convex set. We give a proof of our main result, Theorem \ref{main thm}. 
%In an h-convex open set $\Omega$ of the Heisenberg group $\H$, for a function $f:\Omega\to \mathbb{R}$, 
%One may consider $L$ as a counterpart of the operator $L_{eucl}$. %For a function $f\in C^2(\Omega)$ and $p\in \Omega$, 
%Indeed, we can  it is natural to expect that the sign of $L[f]$ is closely related to the h-quasiconvexity of $f$. 

The relation between $L[f]$ and  h-quasiconvexity of $f$ has been revealed in \cite{CCP2}, but only restricted to the functions $f$ of class $C^2$. We intend to provide a further discussion for upper semicontinuous h-quasiconvex functions and point out differences from the Euclidean cases. 
% It is known that such an operator is closely related to the h-convexity of level sets of 
%Related results in the Euclidean space are provided in 
 %Our examples will show that direct generalizations of sufficient conditions in the Euclidean space does not imply h-quasiconvexity in the Heisenberg group. Instead, a stronger condition on $L$ is provided to guarantee the h-quasiconvexity. Finally, we introduce a stronger notion, which we call uniform h-quasiconvexity and study its relation with the nonlinear elliptic operator $L$.

%which is an analogue of this operator in $\R^n$ used to study the Euclidean quasiconvexity in the previous section. 

Let us review the definition of h-convex sets and h-quasiconvex functions. 
\begin{defi}[H-convex sets]
A set $E \subset \H$ is said to be an h-convex set in $\H$, if for every $p \in E$ and $q \in E \cap \H_p$, the horizontal segment $[p,q] := \{\lambda p + (1 - \lambda ) q: \lambda \in [0,1]\}$ stays in $E$. 
\end{defi}

\begin{defi}[H-quasiconvex functions]\label{hqc}
Let $\Omega$ be an h-convex set in $\H$. We say a function $f$ is h-quasiconvex in $\Omega$ if 
\begin{equation}\label{def-qc}
\begin{aligned}
f(w)\leq \max & \{f(p),   f(q)\} \quad  \text{for all $p\in \Omega$, $q \in \H_p \cap \Omega$ and $w \in [p,q]$,}
\end{aligned}
\end{equation}
or, equivalently, all sublevel sets of $f$ are h-convex subsets in $\Omega$.
\end{defi}

%of upper semicontinuous functions $f$.

% some sufficient conditions and necessary conditions involving $L_{eucl}$ are given in \cite{BGJ13}.

%and 
%$$\hat{L}(f)(p)=L^*(\nabla_H f(p), (\nabla_H^2 f(p))^*).$$

We next extend Definition \ref{sub} to the sub-Riemannian case by replacing  $L_{eucl}$,  $\nabla \varphi, \nabla^2 \varphi$ respectively by $L$, $\nabla_H \varphi, (\nabla_H^2 \varphi)^\star$. For our later application, we present a generalized form with a general constant on the right hand side. 

\begin{defi}[Subsolutions associated to h-quasiconvexity operator]\label{sub heis s}
Let $a\in \R$. A locally bounded function $f\in USC({\Omega})$ is said to be a viscosity subsolution of $L[f]\ge a$ (resp., $L[f]> a$) in $\Omega$, if whenever $f-\varphi$ achieves a strict local maximum at $\hat{p}\in \Omega$ for a smooth function $\varphi:\Omega\to \mathbb{R}$, we have
\begin{equation}\label{eq sub heis0}
-L[\varphi](\hat{p})\le -a \quad (\text{resp., } \ \  -L[\varphi](\hat{p})<-a).
\end{equation}
We say $f$ satisfies $L[f]\geq a$ (resp., $L[f]>a$) in the viscosity sense if it is a viscosity subsolution of $L[f]\geq a$ (resp., $L[f]>a$).
\end{defi}

One can consider a weaker variant of the definition by adopting the upper semicontinuous envelope $L^\ast$ in \eqref{envelope1}. Note that $L^\ast[\varphi]$ can also be expressed via $L_0^\ast$ with horizontal gradient and Hessian for $n=2$, i.e., 
\[
L^\ast[\varphi](p):=L_0^\ast(\nabla_H \varphi(p), (\nabla_H^2\varphi)^\star(p)).
\]

\begin{defi}[Subsolutions associated to the operator envelope]\label{sub heis}
Let $a\in \R$. A locally bounded function $f\in USC({\Omega})$ is said to be a viscosity subsolution of $L^\ast[f]\ge a$ (resp., $L^\ast[f]> a$) in $\Omega$, if whenever $f-\varphi$ achieves a strict local maximum at $\hat{p}\in \Omega$ for a smooth function $\varphi:\Omega\to \mathbb{R}$, we have
\begin{equation}\label{eq sub heis}
-L^\ast[\varphi](\hat{p})\le -a \quad (\text{resp., } \ \  -L^\ast[\varphi](\hat{p})<-a),
\end{equation}
where 

We say $f$ satisfies $L^\ast[f]\geq a$ (resp., $L^\ast[f]>a$) in the viscosity sense if it is a viscosity subsolution of $L^\ast[f]\geq a$ (resp., $L^\ast[f]>a$).
\end{defi}
Note that, in view of \cite[Remark 2.5]{BGJ13}, $L^\ast[\varphi](p)$ is the maximum eigenvalue of $(\nabla_H^2\varphi)^\star(p)$ if $\nabla_H\varphi(p)=0$.

In addition to the above two types of subsolutions, we introduce a third intermediate notion using the function envelope $\ol{L[f]}$ defined as in \eqref{envelope2} for $f\in C^2(\Omega)$ and $p\in \Omega$. Note that in general it is not equal to $L[f]$ or  $L^\ast[f]$. % of the operator $L$ applied to $f$. 
In fact, we have \eqref{compare} for $f: \Omega\to \R$ smooth and $p\in \Omega$.
\begin{defi}[Subsolutions associated to the function envelope]
Let $a\in \R$. A locally bounded function $f\in USC({\Omega})$ is said to be a viscosity subsolution of $\ol{L[f]}\ge a$ (resp., $\ol{L[f]}> a$) in $\Omega$, if whenever $f-\varphi$ achieves a strict local maximum at $\hat{p}\in \Omega$ for a smooth function $\varphi:\Omega\to \mathbb{R}$, we have
\[
-\ol{L[\varphi]}(\hat{p})\le -a \quad (\text{resp., } \ \  -\ol{L[\varphi]}(\hat{p})<-a).
\]
We say $f$ satisfies $\ol{L[f]}\ge a$ (resp., $\ol{L[f]}> a$) in the viscosity sense if it is a viscosity subsolution of $\ol{L[f]}\ge a$ (resp., $\ol{L[f]}> a$). 
\end{defi}

%for $f:\H\to \mathbb{R}$, let $\ol{L[f]}$ denote the upper semicontinuous envelope of the function $L[f]:\H\to \mathbb{R}$, that is,

%{\color{blue}With $L[f]$ and $L[\varphi]$ replaced by $\ol{L[f]}$ and $\ol{L[\varphi]}$ respectively, we can define $f$ satisfies $\ol{L[f]} \ge a$ (resp., $\ol{L[f]} > a$) in the viscosity sense in $\Omega$ by Definition \ref{sub heis s}.}

From the viewpoint of standard viscosity solution theory in the Euclidean space, the weakest notion in Definition \ref{sub heis} seems to be the most suitable option to understand a subsolution of $L[f]\geq a$ in $\Omega$. 
However, as already emphasized before, in the Heisenberg group these notions, especially Definition \ref{sub heis s} and Definition \ref{sub heis}, demonstrate distinct properties when we use them to characterize h-quasiconvex functions. We therefore keep all these separate definitions in our work for different purposes and carefully distinguish the terminology about them in our later use.

 %{\color{yellow}We also give a characterization of uniformly h-convex sets in terms of uniformly h-quasiconvex functions.} 

\subsection{Necessary condition for h-quasiconvexity}\label{sec:necessary}

We prove the first statement in Theorem \ref{main thm}, which is analogous to Theorem \ref{necessary} for the Euclidean case. 

%{\color{blue} In the following we introduce a similar definition for the purpose of the sufficient condition to be h-quasiconvex. 
%\begin{defi}\label{sub heis s}
%Let $a\in \R$. A locally bounded function $f\in USC({\Omega})$ satisfies $L[f]\ge a$  (resp., $L[f]> a$) in the viscosity sense in $\Omega$, if whenever $f-\varphi$ achieves a strict local maximum at $\hat{p}\in \Omega$ for a smooth function $\varphi:\Omega\to \mathbb{R}$, we have
%\begin{equation}\label{eq sub heis s}
%-L[\varphi](\hat{p})\le -a \quad (\text{resp., } \ \  -L[\varphi](\hat{p})<-a).
%\end{equation}
%\end{defi}

%Similar definition applies for $L^*$ (with $L$ replaced by $L^*$ in the definition). We remark that being a viscosity subsolution of $L[f]\ge a$ (resp., $L[f]> a$) in $\Omega$ is equivalent to $L^*[f]\ge a$  (resp., $L^*[f]> a$) in the viscosity sense in $\Omega$. Since $L^*[f] \ge L[f]$, $L[f]\ge a$  (resp., $L[f]> a$) in the viscosity sense implies being a viscosity subsolution of $L[f]\ge a$ (resp., $L[f]> a$). We should be careful about the terminology here.

%Analogously to  the viscosity subsolution property %the viscosity inequality 
%${L}[f]\ge 0$ is a necessary condition for $f$ to be h-quasiconvexity. 
\begin{thm}[Necessary condition for h-quasiconvexity]\label{strict heis}
Let $\Omega$ be an h-convex open set in $\H$. If $f \in USC(\Omega)$ is h-quasiconvex, then %$f$ {\color{blue}satisfies $L[f]\ge 0$ in the viscosity sense in $\Omega$. In particular, 
$f$ is a viscosity subsolution of $L[f]\geq 0$ in $\Omega$. In particular, $f$ is also a viscosity subsolution of $L^\ast[f]\geq 0$ in $\Omega$. 
\end{thm}

\begin{proof}
Suppose that $f:\Omega\to \R$ is h-quasiconvex but  $f$ fails to satisfy $L[f]\geq 0$ in the viscosity sense. Then there exist a smooth function $\varphi:\Omega\to \R$ and $\hat{p}\in \Omega$ such that $f-\varphi$ achieves a maximum at $\hat{p}$ and 
\[
L(\varphi)(\hat{p})=\min\{\langle (\nabla^2_H \varphi)^\star(\hat{p})\eta, \eta\rangle: \eta\in \mathbb{R}^2, |\eta|=1, \langle \nabla_H \varphi(\hat{p}),\eta\rangle=0\}<0.
\]  
It follows that there exists a unit vector $v_h\in \mathbb{R}^2$ such that $\langle\nabla_H \varphi(\hat{p}), v_h\rangle = 0$ and $\langle (\nabla^2_H \varphi)^\star(\hat{p})v_h, v_h\rangle=-c$ with $c>0$. By Taylor expansion, for $r>0$ sufficiently small and $v=(v_h, 0)\in \H_0$, we have
\[
\begin{aligned}
f(\hat{p}\cdot rv)\le \varphi(\hat{p}\cdot r v)&=\varphi(\hat{p})+r\langle\nabla_H \varphi(\hat{p}), v_h\rangle+\frac{r^2}{2}\langle (\nabla^2_H \varphi)^\star(\hat{p})v_h, v_h\rangle+o(r^2)\\
&=f(\hat{p})-\frac{c r^2}{2}+o(r^2),
\end{aligned}
\]
\[
\begin{aligned}
f(\hat{p}\cdot rv^{-1})\le \varphi(\hat{p}\cdot rv^{-1})&=\varphi(\hat{p})-r\langle\nabla_H \varphi(\hat{p}), v_h\rangle+\frac{r^2}{2}\langle (\nabla^2_H \varphi)^\star(\hat{p})v_h, v_h\rangle+o(r^2)\\
&=f(\hat{p})-\frac{c r^2}{2}+o(r^2).
\end{aligned}
\]
On the other hand, since $f$ is h-quasiconvex, we have
\[
f(\hat{p})\le \max\{f(\hat{p}\cdot r v), f(\hat{p}\cdot rv^{-1})\}\le f(\hat{p})-\frac{c r^2}{2}+o(r^2).
\]
Dividing $r^2$ on both sides and passing to the limit as $r\to 0$ yield a contradiction. 
\end{proof}

%Since $L^\ast[f]\ge L(f)\ge 0$, we also get that h-quasiconvexity implies ${L}(f)\ge 0$ in the viscosity sense, or $\hat{L}(f)\ge 0$. 
The following sub-Riemannian variant of  \cite[Example 1.1]{BGJ13} shows that $L[f]\ge 0$ in the viscosity sense does not imply h-quasiconvexity of $f$. 

\begin{example}
Let $g: \R\to \R$ be given by $g(t)=-t^4$. We show that the smooth function $f(x,y,z) =g(z)$ is not h-quasiconvex but $L^\ast[f] = L[f] =0$ in $\H$. Indeed, a direct computation yields, for $p=(x, y, z)\in \H$, 
\[
\nabla_H f(p) = (X_1 f (p) ,X_2 f (p) ) = \frac{g'(z)}{2} (-y, x)
\]
and
\begin{align*}
(\nabla_{H}^2 f)^\star(p) = \frac{g''(z)}{4}\begin{pmatrix}
y^2 & -xy \\
-xy & x^2 
\end{pmatrix}.
\end{align*}
Here $g'(z) = - 4 z^3$ and $g''(z) = -12 z^2$. We divide our argument into two cases. Suppose that $\nabla_H f (p) = 0$. Then either $z = 0$ or $(x, y) = 0$ holds. It follows from either of the conditions that $(\nabla_{H}^2 f)^\star(p)= 0$ and thus $L^\ast[f](p)= L[f](p)=0$.  If $\nabla_H f(p) \ne 0$, then taking $\eta=(x, y)/\sqrt{x^2+y^2}$, we immediately deduce that 
\[
 L^\ast[f](p) = L[f](p) =\la (\nabla_H^2 f)^\star(p)\eta, \eta\ra=0.
\]
Finally, to show that $f$ is not h-quasiconvex, it suffices to notice that for $p_1=(1, 2, 1)$, $p_2=(1, -2, -1)$ and $q=(1, 0, 0)$, 
\[
f(q) = 0 > -1 = f(p_1) = f(p_2)
\]
and $p_2, q\in \H_{p_1}$.
%\[
%(1,-2,-1) = (1,2,1) \cdot (0,-4,0),
%\]
%which implies $(1,-2,-1) \in \H_{(1,2,1)}$.
\end{example}

%One may expect that Theorem \ref{strict} and Theorem \ref{localmax} can also be generalized directly for h-quasiconvexity in $\H$. 

%The example below shows otherwise. 

\subsection{Sufficient condition for h-quasiconvexity} \label{sec:sufficient}
We next prove the second statement of Theorem \ref{main thm}, which generalizes Theorem \ref{strict} in our sub-Riemannian setting. The strict inequality $L[f]>0$ or $\ol{L[f]}> 0$ is known to be a sufficient condition for h-quasiconvexity of $f$, as verified by Calogero, Carcano and Pini \cite[Theorem 4.6]{CCP2} only for $f\in C^2(\Omega)$. Our result below further extends their result in the viscosity sense to the class of upper semicontinuous functions. 

\begin{thm}[Sufficient condition for h-quasiconvexity]\label{sufficient1-heis}
    Let $\Omega$ be an h-convex open set in $\mathbb{H}$. If $f\in USC(\Omega)$ is a viscosity subsolution of $\ol{L[f]}> 0$ in $\Omega$, %; namely, we have
%\begin{equation}\label{eq sufficient1}
%-\ol{L[\varphi]}(\hat{p}) < 0
%\end{equation}
%whenever $f-\varphi$ achieves a strict local maximum at $\hat{p}\in \Omega$ for a smooth function $\varphi:\Omega\to \mathbb{R}$. 
then $f$ is h-quasiconvex in $\Omega$. In particular, if $f\in USC(\Omega)$ is a viscosity subsolution of $L[f]>0$ in $\Omega$, then $f$ is h-quasiconvex in $\Omega$.
   % \[ L(u)=L(\nabla_H u(p), (\nabla^2_H u(p))^*) := \min \{ v \cdot (\nabla^2_H u(p))^* \cdot v^T: |v|=1, \langle v, \nabla_H u(p) \rangle =0  \} >0,\]
\end{thm}
\begin{proof}
    Suppose that $f$ is not h-quasiconvex. 
% To avoid repeatedly using the letter $p$
Then there exist $q_1,q_2 \in \Omega$, $q_2 \in \mathbb{H}_{q_1}$ and $w \in [q_1,q_2]$ such that $f(q_1)\leq f(q_2) < f(w)$. %Without the loss of generality, (
 By left translation by $q_1^{-1}$ and dilation $\delta_\ell$ with $\ell={|q_1^{-1}\cdot q_2|_G}$, we can assume that $q_1=(0,0,0)$, $q_2=(a,b,0)$ and $w=\alpha q_2$ for some $a,b \in \mathbb{R}$, $a^2+b^2=1$ and $\alpha \in (0,1)$. 
Let $\pi$ denote the Euclidean projection onto the plane $ax+by=0$, that is, 
\[
\pi(x, y, z)=\left(b^2 x-aby, -abx+a^2 y, z\right).
\]
%Then there exist $p,q \in \Omega$, $q \in \mathbb{H}_p$ and $w \in [p,q]$ such that $f(p)\leq f(q) < f(w)$. %Without the loss of generality, (
% By left translation by $p^{-1}$ and dilation $\delta_\ell$ with $\ell={|p^{-1}\cdot q|_G}$, we can assume that $p=(0,0,0)$, $q=(a,b,0)$ and $w=\alpha q$ for some $a,b \in \mathbb{R}$, $a^2+b^2=1$ and $\alpha \in (0,1)$. 
   In view of the upper semicontinuity of $f$, there exists a closed disk $D_r$ in the plane $ax+by=0$ centered at $0$ with radius $r>0$ small such that 
\[
Q_r:=\{(x,y,z)\in \H: \langle(x,y), (a,b)\rangle \in [0,1], \pi(x,y,z) \in D_r \} \subset \Omega   %Q:=\{(x,y,z)\in \H: (x,y)=s(a,b) \text{ for } s\in [0,1], \pi(x,y,z) \in D_r \} \subset \Omega
\] 
and $f(p) < f(w)$ for all $p=(x,y,z)\in Q_r$ satisfying $ax+by=0$ or $ax+by=1$.

 %and
%    for $(x,y,z) \in T$ with either $ax+by=0$ or $ax+by=1$.

    Let $\varphi:\Omega \to \mathbb{R}$ be defined by
    \[ \varphi(x,y,z) := f(w) + k \left((bx-ay)^2 + \left(z-\frac{(ax+by)(bx-ay)}{2}\right)^2\right)\]
    for $k>0$ large enough to have $\varphi>f$ on $\partial Q_r$ and $f(w)=\varphi(w)$. Then, $f-\varphi$ attains its maximum at some interior point $\hat{p} \in Q_r$, where by assumption $- \ol{L[\varphi]}(\hat{p}) < 0$ holds.

    Notice that $\la(a,b), (X_1\varphi(p),X_2\varphi(p)) \ra= 0$ for any $p=(x,y,z)\in \Omega$. In fact, by direct computation, one obtains that
   \begin{align*}
       X_1 \varphi(p) &= 2bk(bx-ay)-kZ (2abx-a^2y+b^2y)-kyZ, \\
       X_2 \varphi(p) &= -2ak(bx-ay)-kZ(-a^2x+b^2x-2aby) + kxZ  
   \end{align*} 
   with
   \[
   Z=z-\frac{(ax+by)(bx-ay)}{2}.
   \]
  Then it follows that
   \begin{align*}
a  X_1 \varphi(p) + bX_2 \varphi(p) 
     &= kZ(-2a^2bx+a^3y-ab^2y-ay)+kZ(a^2bx-b^3x+2ab^2y+bx)\\
        &=kZ\left(-a^2bx+a^3y+ab^2y-b^3x+(bx-ay)\right)\\
        &=kZ \left(a^2(ay-bx)+b^2(ay-bx)+(bx-ay)\right)=0.
   \end{align*} 
   The last equality follows from the fact that $a^2+b^2=1$. Therefore, for each $p\in \Omega$, we can plug $\eta=(a, b)$ into the definition of $L[\varphi](p)$ in \eqref{sec operator heis} and obtain
    \begin{align*}
        L[\varphi](p) \leq \la (\nabla_H^2 \varphi)^\star(p) \eta, \eta \ra & = a^2 X_1^2\varphi(p) + ab (X_1X_2\varphi(p) +X_2X_1\varphi(p))+ b^2 X_2^2\varphi(p) \\ 
        &= (aX_1+bX_2)^2\varphi(p) =0. 
    \end{align*}
 It follows that $L[\varphi](\hat{p})\leq 0$ and
 \[
 \ol{L[\varphi]}(\hat{p})=\limsup_{p\to \hat{p}}L[\varphi](p)\leq 0,
 \] 
 which contradicts $- \ol{L[\varphi]}(\hat{p}) < 0$.
\end{proof}

\begin{rmk}
The condition that $\overline{L[f]}>0$ is only a sufficient condition to guarantee h-quasiconvexity but not a necessary condition. Note that any constant function is obviously h-quasiconvex in $\H$ but fails to satisfy this strict inequality in the viscosity sense. %viscosity supersolution property. 
\end{rmk}

Although a sub-Riemannian variant of Theorem \ref{localmax} may be expected to hold in the Heisenberg group as well, it is not clear whether the conditions that $L[f]\geq 0$ and $f$ does not attain local maxima imply h-quasiconvexity of $f$. Our example below shows that such implication fails to hold if one changes the subsolution condition to $L^\ast[f]\geq 0$, or $\ol{L[f]}\geq 0$ or even a strict inequality $L^\ast[f]> 0$.

\begin{example}\label{positive}
Let $f:\H\to\R$ be given by $$f(x,y,z) := x^2 + \left(z +{ xy\over 2}\right)^2.$$ We shall prove that $f$ is not h-quasiconvex, $f$ does not attain local maxima, and $L^\ast[f] > 0$ and $\ol{L[f]}\geq 0$ hold everywhere in $\H$.

Let $p_1=(-\vep,1, 1)$ and $p_2=(\vep, -1,1)$ for $\vep > 0$ small enough.  Take $q=(0, 0, 1)$. We have $p_1, q\in \H_{p_2}$ and
\[
f(q) = 1 > \vep^2 + \left( 1 - \frac{\vep}{2}\right)^2 =   f(p_1) = f(p_2).
\]
It then follows that $f$ is not h-quasiconvex. 

Moreover, by direct computations, $\nabla f(p)=0$ only holds at $p=(x,y,z)$ satisfying $x=z=0$. However, it is not difficult to see that the function has a minimum value $0$ on the $y$-axis. Therefore $f$ cannot attain local maxima anywhere in $\H$.

In addition, we can also compute the derivatives to show at $p=(x, y, z)$,
\[
\nabla_H f (p)  = (2x, 2x(z + xy/2)),
\]
which implies $\nabla_H f (p) = 0$ if and only if $x = 0$, and
\begin{align*}
(\nabla_{H}^2 f)^\star (p)= \begin{pmatrix}
2 & z + xy/2 \\
z + xy/2 & 2 x^2 
\end{pmatrix}.
\end{align*}
We discuss the vanishing and non-vanishing gradient cases separately. % into two cases to check $\hat{L}(f) > 0$. 

In the case $\nabla_H f (p) = 0$ or equivalently $x = 0$, we observe that
\[
\tr (\nabla_{H}^2 f)^\star(p)=2+2x^2>0
\] 
holds, which implies that the maximum eigenvalue of $(\nabla_{H}^2 f)^\star (p)$ is positive as well. It follows that  $L^\ast[f](p)>0$ in this case.

{Suppose now that $\nabla_H f (p) \ne 0$ or equivalently $x \ne 0$.}
In this case, letting
\[
\eta = \frac{1}{\sqrt{1 + (z + xy/2)^2}}(z + xy/2, - 1),
\]
 we get $\la \nabla_H f (p), \eta\ra = 0$, $|\eta| = 1$ and
\[
\la(\nabla_{H}^2 f)^\star(p)\eta,  \eta\ra= \frac{1}{1 + (z + xy/2)^2} (2 (z + xy/2)^2 - 2(z + xy/2)^2 + 2x^2)> 0,
\]
which implies $L^\ast[f](p)=L[f](p)>0$. Combining these two cases, we see that $\ol{L[f]}\geq 0$ also holds everywhere in $\H$. 
The proof of our claim is now complete. 

Note that in the vanishing gradient case above, we cannot obtain $L[f](p)\geq 0$ but only $L^\ast[f](p)\geq 0$ and $\ol{L[f]}(p)\geq 0$. It would be interesting to further investigate the analogue of Theorem \ref{localmax} in the case of Heisenberg group. 

This example also indicates a significant difference between the Euclidean and sub-Riemannian situations about the adoption of semicontinuous envelopes in the definition of subsolutions. Recall that, as clarified in Remark \ref{rmk weak} and Theorem \ref{localmax2}, the weaker viscosity inequality $L_{eucl}^\ast(f)\geq 0$, together with the absence of local maxima, is already sufficient to guarantee the quasiconvexity of $f$ in the Euclidean space. However, in the Heisenberg group, Example \ref{positive} shows that having $L^\ast(f)\geq 0$ or $\ol{L[f]}\geq 0$ is not enough to get h-quasiconvexity of $f$ even if there exist no local maxima of $f$. 
\end{example}

%{\color{blue}For the sake of our application to convexity preserving for horizontal curvature flow in the following section, we introduce the following definition.} 

%that $L^*[f]>0$ in the viscosity sense does not guarantee the h-quasiconvexity of $f$. 

%\begin{rmk}
%in the following two aspects: first, it relaxes the regularity of $f$ from $C^2$ to the class of upper semicontinuous; second, when $\nabla f(p)\neq 0$ for $f\in C^2$, 

%\end{rmk}

For our later application, let us apply the sufficient condition in Theorem \ref{sufficient1-heis} to a special case when the level sets of $f$ are rotationally  symmetric about $z$-axis. 

\begin{example}\label{ex rot-symmetry}
We study the h-quasiconvexity of $f$ when $f(p)=r^2-g(z)$ where $r=(x^2+y^2)^{1/2}$ for $p=(x, y, z)\in \H$ and $g\in C(\R)$ is a given function. In this case, the sublevel set $E=\{f<0\}$ is rotationally symmetric with respect to the $z$-axis. In other words, $f$ is a function whose $0$-level set is a surface of revolution $\partial E$ generated by rotating the graph of $r=g(z)^{1/2}$ around the $z$-axis.

If $g$ is assumed to be of class $C^2$, then we can calculate $L[f]$ as follows. Note that by direct computation we have
\[
\nabla_H f(p)=\left(2x+{y\over 2} g'(z),  2y-{x\over 2} g'(z)\right),
\]
\[
|\nabla_H f(p)|={1\over 2}(x^2+y^2)^{1\over 2} (16+g'(z)^2)^{1\over 2}={1\over 2} r(16+g'(z)^2)^{1\over 2},
\]
\[
(\nabla_H^2 f)^\star(p)=\begin{pmatrix}
2-y^2 g''(z)/4 & xy g''(z)/4\\
 xy g''(z)/4 & 2-x^2 g''(z)/4 
\end{pmatrix}.
\]
If $\nabla_H f(p)\neq 0$, i.e., $(x, y)\neq (0, 0)$, then using 
\[
\eta={1\over (x^2+y^2)^{1\over 2} (16+g'(z)^2)^{1\over 2}}\left(4y-{x} g'(z), -4x-{y} g'(z)\right),
\] 
we get
\[
L[f](p)=\la (\nabla_H^2 f)^\star(p) \eta, \eta \ra=2-{4r^2 g''(z)\over 16+g'(z)^2}.
\]
If $\nabla_H f(p)= 0$, i.e., $(x, y)= (0, 0)$, then we have $(\nabla_H^2 f)^\star(p)= 2I$ and thus $L[f](p)=2$. Hence, $f$ is h-quasiconvex in an h-convex open set $\Omega\subset \H$ if 
\begin{equation}\label{rot graph cond}
1-{2(x^2+y^2) g''(z)\over 16+g'(z)^2}>0\quad \text{for all $(x, y, z)\in \Omega$.}
\end{equation}
In order for the argument here to work, $g$ actually need not be of class $C^2(\R)$. By Theorem \ref{sufficient1-heis} we see that $f$ is h-quasiconvex in $\Omega$ if $g\in LSC(\R)$ and \eqref{rot graph cond} holds only in the viscosity sense.

A more general situation is the case when $f(p)=F(\rho, z)$, where $\rho=r^2=x^2+y^2$ and $F$ is a function in $[0, \infty)\times \R$. In this case, by similar computations, we have 
\[
L[f]=2F_\rho+4\rho \frac{F_{\rho\rho} F_z^2-2F_{\rho z}F_\rho F_z+F_{zz}F_\rho^2}{F_z^2+16 F_\rho^2} 
\]
at any point $p\in \H$ satisfying $F_\rho(p)\neq 0$ or $F_z(p) \neq 0$.

\end{example}

%\begin{example}\label{ex rot-symmetry2}
%For a bounded nonempty open set $E\subset \H$ enclosed by a surface of revolution $\partial E$, it is more involved to construct a function $f\in C(\H)$ that satisfies $E=\{f<0\}$ and $L[f]>0$ in $\H$. We need to modify Example \ref{ex rot-symmetry}; see Proposition \ref{prop rot symmetry} for details. % and add more conditions on the function $g$. 

%We immediately get $\hat{L}(f)\ge L(f)^\star\ge 0$ when $f$ is h-quasiconvex from \ref{Hnecessary}. 

\subsection{Uniformly h-quasiconvex functions}\label{sec:uniform}

This section is devoted to an even stronger notion of h-quasiconvex funtions, which we call uniform h-quasiconvexity.  
\begin{defi}[Uniformly h-quasiconvex functions]\label{uniformqc}
Let $\Omega\subset \H$ be an h-convex set. We say a function $f$ is uniformly h-quasiconvex in $\Omega$ if there exists  $r_0>0$ and 
$\lambda>0$ such that 
\begin{equation}\label{uni-qc}
\begin{aligned}
f(p)\leq \max & \{f(p\cdot rv),   f(p\cdot rv^{-1})\} -\lambda r^2\\
& \text{for all $0 < r < r_0$, $p\in \Omega$, $v\in \H_0$ with $|v|=1$ such that $p\cdot rv, p\cdot rv^{-1}\in \Omega$.}
\end{aligned}
\end{equation}
%When $\Omega=\H$, we simply say $f$ is uniformly h-quasiconvex. 
\end{defi}
By definition one cannot show directly that uniformly h-quasiconvex functions are all h-quasiconvex. But for locally bounded USC functions we can show in Proposition \ref{prop uqc vis} that they are actually  h-quasiconvex.

%\begin{defi}\label{uniformqc}
%Let $\Omega$ be an h-convex set in $\H$. We say a locally bounded function $f\in USC(\H)$ is uniformly h-quasiconvex in $\Omega$ if there exists  %$r_0>0$ and 
%$\lambda>0$ such that 
%\begin{equation}\label{uni-qc}
%\begin{aligned}
%f(p)\leq \max & \{f(p\cdot rv),   f(p\cdot rv^{-1})\} -\lambda r^2\\
%& \text{for all $r>0$, $p\in \Omega$, $v\in \H_0$ with $|v|=1$ such that $p\cdot rv, p\cdot rv^{-1}%\in \Omega$.}
%\end{aligned}
%\end{equation}
%%When $\Omega=\H$, we simply say $f$ is uniformly h-quasiconvex. 
%\end{defi}
%%
%By definition one can show that uniformly h-quasiconvex functions are all h-quasiconvex. 
It is thus easily seen that in a convex domain $\Omega\subset \R^3$, all uniformly convex functions in the Euclidean sense are uniformly h-quasiconvex. On the other hand, a function that is not uniformly convex in $\R^3$ can still be uniformly h-quasiconvex in $\H$, as shown by the following example. 

\begin{example} \label{ex linear}
Consider again the rotational symmetric case
\begin{equation}\label{rot-graph}
f(x, y, z)=x^2+y^2-g(z)\quad \text{for $(x, y, z)\in \H$}
\end{equation}
with $g\in C(\R)$. If we take $g$ to be a {concave} function, then for any $p=(x, y, z)\in \H$, $v=(\eta_1, \eta_2, 0)
\in \H_0$ with $|v|=1$ and $r>0$, since 
\[- {1\over 2}g\left(z+{1\over 2}r x \eta_2-{1\over 2}r y \eta_1\right)-{1\over 2}g\left(z-\left({1\over 2}r x \eta_2-{1\over 2}r y \eta_1\right)\right) \ge -g(z),
\] by concavity of $g$, we have
\[
\begin{aligned}
& \max \{f(p\cdot r v), f(p\cdot rv^{-1})\}\geq {1\over 2} f(p\cdot r v)+{1\over 2} f(p\cdot rv^{-1})\\
&\geq  {1\over 2}(x+r\eta_1)^2+ {1\over 2}(y+r\eta_2)^2+ {1\over 2}(x-r\eta_1)^2+ {1\over 2}(y-r\eta_2)^2-g(z)\\
&\geq x^2+y^2-g(z)+r^2=f(x, y, z)+r^2,
\end{aligned}
\]
which shows that $f$ is uniformly h-quasiconvex in $\H$. In particular, if we choose $g$ to be a linear function, then $f$ is not uniformly convex in $\R^3$.

In fact, $f$ can be uniformly h-quasiconvex in $\H$ even if it is not convex in $\R^3$. Assume that $g$ satisfies $g''\leq C_1$ in the viscosity sense, or equivalently $g$ is a semiconcave function with semiconcavity constant $C_1$; see for example \cite{CaSBook} for the definition of semiconvex/semiconcave functions and \cite{ALL, Im1, BaDr1} for the viscosity characterization.  %\textcolor{blue}{semiconvexity references, delete this blue comment after confirmation} 
Let
\[
\Omega=\{(x, y, z): x^2+y^2<C_2\}
\]
with $C_1, C_2>0$ fulfilling $C_1C_2<8$. Then the same calculation as above yields
\[
\begin{aligned}
& \max \{f(p\cdot r v), f(p\cdot r v^{-1})\}\\
 &\geq x^2+y^2+ r^2 - {1\over 2}g\left(z+{1\over 2}r x \eta_2-{1\over 2}r y \eta_1\right)-{1\over 2}g\left(z-\left({1\over 2}r x \eta_2-{1\over 2}r y \eta_1\right)\right)\\
 & \geq x^2+y^2-g(z)+r^2-{C_1\over 8} r^2 (x^2+y^2) \geq f(p)+\left(1-{C_1C_2\over 8}\right)r^2
 \end{aligned}
\]
for all $p=(x, y, z)\in \H$, $v\in \H_0$ with $|v|=1$ and $r>0$ such that $p\cdot r v, p\cdot r v^{-1}\in \Omega$. Hence, in this case $f$ is uniformly h-quasiconvex in $\Omega$ but $f$ is not necessarily convex in the Euclidean sense. 

\end{example}

Using the elliptic operator $L[f]$ in \eqref{sec operator heis}, 
%\[
%L(f)(p)=\min\{\la (\nabla_H^2 f(p))^\ast \eta, \eta\ra: |\eta|=1, \la \nabla_H f(p), \eta\ra=0\},
%\]
we can establish a necessary condition for $f$ to be uniformly h-quasiconvex, as below. 

\begin{prop}[Necessary condition for uniform h-quasiconvexity]\label{prop uqc vis}
Let $\Omega$ be an open set in $\H$. Let $f\in USC(\Omega)$ be locally bounded and satisfies \eqref{uni-qc} for some %$r_0>0$ and 
$\lambda>0$. Then $f$ is a viscosity subsolution of $L[f]\geq 2\lambda$ in $\Omega$. In particular, $f$ is also a viscosity subsolution of $\ol{L[f]}\geq 2\lambda$ in $\Omega$.
%\begin{equation}\label{uni-qc-vis}
%-L[f]\leq -2\lambda\quad \text{in $\Omega$}
%\end{equation}
%in the viscosity sense; namely, 
%\begin{equation}\label{uni-qc-vis2}
%-L[\varphi](\hat{p})\leq -2\lambda 
%\end{equation}
%holds whenever $f-\varphi$ achieves a strict local maximum at $\hat{p}\in \Omega$ for a smooth function $\varphi:\Omega\to \mathbb{R}$. In particular, $f$ is a viscosity subsolution of $\ol{L[f]}\geq 2\lambda$ in $\Omega$.
\end{prop}
\begin{proof}
Suppose that there exist a smooth function $\varphi: \Omega\to \R$ and $\hat{p}\in \Omega$ such that $f-\varphi$ attains a local maximum at $\hat{p}$. Then it follows from \eqref{uni-qc} that 
\begin{equation}\label{eq:uqc1}
\varphi(\hat{p})\leq \max\{\varphi(\hat{p}\cdot rv), \varphi(\hat{p}\cdot rv^{-1})\}-\lambda r^2
\end{equation}
for all $r>0$ small and all $v\in \H_0$ with $|v|=1$. Write $v = (v_h,0) \in \H$ for $v_h\in \R^2$.
Then by Taylor expansion, \eqref{eq:uqc1} implies 
\begin{equation}\label{eq:uqc2}
-r|\la \nabla_H \varphi(\hat{p}), v_h\ra |-{r^2\over 2} \la (\nabla_H^2 \varphi)^\star(\hat{p}) v_h, v_h\ra\leq -\lambda r^2 +o(r^2)
\end{equation}
for $r>0$ small and all $v_h\in \R^2$ with $|v_h|=1$. %denotes the horizontal component of $v\in \H_0$. 
If $|\nabla_H \varphi(\hat{p})|\neq 0$, then we take $v_h\in \R^2$ such that $|v_h|=1$ and $\la v_h, \nabla_H \varphi(\hat{p})\ra=0$. This choice yields
\[
-{r^2\over 2} L[\varphi](\hat{p})%\la (\nabla_H^2 \varphi)^\star(\hat{p}) \frac{(\nabla_H \varphi(\hat{p}))^\perp}{|\nabla_H \varphi(\hat{p})|}, \frac{(\nabla_H \varphi(\hat{p}))^\perp}{|\nabla_H \varphi(\hat{p})|}\ra
\leq  -\lambda r^2 +o(r^2).
\]
Dividing the inequality by $r^2$ and sending $r\to 0$, we obtain $-L[\varphi](\hat{p})\leq -2\lambda$.
%\[
%-|\nabla_H \varphi|\dive_H(\nabla_H \varphi/|\nabla_H \varphi|)\leq -2\lambda \quad \text{at $\hat{p}$,}
%\]
%which amounts to saying that \eqref{uni-qc-vis2} holds.

When $\nabla_H \varphi(\hat{p})=0$, by \eqref{eq:uqc2} we immediately get
\[
-{r^2\over 2} L[\varphi](\hat{p})\leq -\lambda r^2+o(r^2).
\]
It is then clear that $-L[\varphi](\hat{p})\leq -2\lambda$ holds again. 
\end{proof}
Together with Theorem \ref{sufficient1-heis}, Proposition \ref{prop uqc vis} shows that uniformly h-quasiconvex functions are h-quasiconvex. We do not know whether or not the reverse implication of Proposition \ref{prop uqc vis} holds.  We have the following result for functions of class $C^2(\Omega)$. 

\begin{prop}[H-quasiconvexity operator on smooth functions]\label{prop c2}
Let $\Omega$ be an open set in $\H$ and $f \in C^2(\Omega)$. Assume that $L[f] \ge 2 \lambda$ holds in $\Omega$ for some $\lambda>0$.  Then, for every compact set $K \subset \Omega$ and $\sigma \in (0,1)$, there exists $r_0 = r_0(K,\sigma) > 0$ such that
\begin{equation}\label{uni-hquasi-c2}
f(p) \le \max\{ f(p \cdot rv) , f(p \cdot r v^{-1})\} - \sigma\lambda r^2
\end{equation}
for all $0 < r < r_0$, $p \in K$ and $v \in \H_0$  with $|v| = 1$.
\end{prop}

\begin{proof}
Fix $p\in K$ arbitrarily. We can find $r_0>0$ small such that $p\cdot r v, p\cdot rv^{-1}$ %\in \Omega$ 
staying in $\Omega$ for all $r\leq r_0$.
Since $f\in C^2(\Omega)$, by Taylor expansion, there exists a continuous increasing function $\omega_K: [0, r_0]\to [0, \infty)$ with $\omega_K(0)=0$ depending on the uniform continuity of $\nabla^2_H f$ in $K$ such that
\begin{equation}\label{uni-hquasi-taylor}
\begin{aligned}
&\left| f(p \cdot rv) -f(p) - r \langle \nabla_H f(p), v_h \rangle - \frac{r^2}{2} \langle  (\nabla_H^2 f)^\star (p) v_h,v_h \rangle\right| \leq  r^2 \omega_{K}(r), \\
&\left| f(p \cdot rv^{-1}) - f(p) + r \langle \nabla_H f(p), v_h \rangle -\frac{r^2}{2} \langle  (\nabla_H^2 f)^\star(p) v_h,v_h \rangle\right| \leq r^2\omega_{K}(r)
\end{aligned}
\end{equation}
 for any $0<r<r_0$ and $v=(v_h, 0)\in \H_0$ with $v_h\in \R^2$ satisfying $|v_h|=1$.
Let us consider the set
\[
\Omega_0 := \{p \in K: (\nabla_H^2 f)^\star (p) > (1+\sigma) \lambda I\},
\]
which is an open subset of $K$ and contains $\{p \in K: \nabla_H f(p) = 0\}$. Suppose that $p \in \Omega_0$. Then by \eqref{uni-hquasi-taylor} we obtain
\[
\max\{ f(p \cdot rv) , f(p \cdot r v^{-1})\} \ge f(p) + \frac{(1+\sigma) \lambda r^2}{2}  - r^2 \omega_{K}(r).
\]
for any $0<r<r_0$ and $v\in \H_0$ with $|v|=1$. Letting $r_0>0$ further small such that $\omega_K(r_0)<(1-\sigma)/2$, we are led to 
the desired inequality \eqref{uni-hquasi-c2}.

If $p\in K \setminus \Omega_0$, then there exists a constant $\varepsilon_0 = \varepsilon_0(K,\sigma) > 0$ such that we have $|\nabla_H f(p)| \ge \varepsilon_0$. %where $v = (v_h,0) \in \H$ with $v_\ast \in \R^2$ as before.
%Now choosing $r_0$ small enough, we obtain the desired uniformly $h$-quasiconvexity for $p \in K_0$. For the other case $ p \in K \setminus K_0$, 
We can write $v_h=v_1 + v_{2}$, where $\la v_1, v_2\ra=\la v_2, \nabla_H f(p)\ra=0$; in other words, $v_h$ is decomposed into the components $v_1$ parallel to $\nabla_H f(p)$ and $v_2$ orthogonal to $\nabla_H f(p)$. It follows immediately that 
\begin{equation}\label{uni-hquasi-taylor2}
|v_1|^2+|v_2|^2=1,
\end{equation}
\begin{equation}\label{uni-hquasi-taylor3}
|\la \nabla_H f(p), v_h \ra|=|\la \nabla_H f(p), v_1\ra| =|\nabla_H f(p)| |v_1|\geq \vep_0 |v_1|
\end{equation}
and 
\[
\la  (\nabla_H^2 f)^\star (p) v_2, v_2 \ra \geq 2\lambda |v_2|^2.
\]
 Note that there exists $C>0$ depending on $\sigma>0$ and  the uniform bound of $\nabla^2_H f$ in $K$ such that 
\[
\left|\la  (\nabla_H^2 f)^\star (p) v_1, v_2 \ra +{1\over 2}\la  (\nabla_H^2 f)^\star (p) v_1, v_1 \ra  \right|\leq \frac{1-\sigma}{2}\lambda |v_2|^2+ C|v_1|^2
\]
for all $v_h = v_1 + v_2 \in \R^2$ with $|v_h|=1$. 
Then \eqref{uni-hquasi-taylor} implies that
\[
\begin{aligned}
f(p \cdot rv) &\geq  f(p) + r \langle \nabla_H f(p), v_1 \rangle+\lambda  r^2 |v_2|^2 -\frac{1-\sigma}{2}\lambda r^2 |v_2|^2 - Cr^2|v_1|^2 - r^2\omega_K(r),\\
f(p \cdot rv^{-1}) &\geq  f(p) - r \langle \nabla_H f(p), v_1 \rangle+\lambda r^2|v_2|^2 -\frac{1-\sigma}{2}\lambda r^2 |v_2|^2- Cr^2|v_1|^2 -r^2\omega_K(r)
\end{aligned}
\]
for all $0<r<r_0$ and $v=(v_h, 0)\in \H_0$ with $|v|=1$. As a result, by \eqref{uni-hquasi-taylor2} and \eqref{uni-hquasi-taylor3} we are led to 
\[
\begin{aligned}
&\max\{ f(p \cdot rv) , f(p \cdot r v^{-1})\} - f(p) \\
& \geq r| \langle \nabla_H f(p), v_1 \rangle|  +\frac{1+\sigma}{2} \lambda r^2 |v_2|^2- Cr^2|v_1|^2 - r^2\omega_K(r)\\
& \geq \sigma \lambda r^2 + \left(\vep_0   - (\sigma \lambda   + C  + \omega_K(r) )r \right) r|v_1| + \left(\frac{1 - \sigma}{2} \lambda - \omega_K(r)\right) r^2 |v_2|^2. \\
%&\geq\sigma \lambda r^2+ \vep_0 r -\left(\sigma \lambda +C\right)r^2- r^2\omega_K(r).
\end{aligned}
\]
Taking $r_0>0$ sufficiently small so that 
\[
\vep_0 \geq \left(\sigma \lambda +C+\omega_K(r_0)\right)r_0 \quad \mbox{and} \quad \frac{1 - \sigma}{2} \lambda \ge \omega_K(r_0),
\]
we obtain \eqref{uni-hquasi-c2} again for any $0<r<r_0$ and $v\in \H_0$ with $|v|=1$. Since our estimates above hold uniformly for all  $p\in K$, we complete the proof. 
\end{proof}

\section{Application to convexity preserving for horizontal curvature flow}

\subsection{The game-theoretic approximation}

For the reader's convenience, we first recall the definition of viscosity solutions of \eqref{mcf} below. We write the horizontal curvature operator as 
\[
F(\xi, X)= -\tr\left(I-{\xi\otimes \xi\over |\xi|^2}\right)X \quad \text{for $(\xi, X)\in \R^2\times \S^2$ with $\xi\neq 0$.}
\]
Let $F^\ast$ and $F_\ast$ denote the upper and lower semicontinuous envelopes in $\R^2\times \S^2$ respectively. 
\begin{defi}[Solutions of horizontal curvature flow]\label{def mcf}
A locally bounded function $u\in USC(\H\times (0, \infty))$ (resp., $u\in LSC(\H\times (0, \infty))$) is said to be a viscosity subsolution (resp., viscosity supersolution) of \eqref{mcf} if whenever $u-\varphi$ achieves a strict local maximum (resp., strict local minimum) at $(\hat{p}, \hat{t})\in \H\times (0, \infty)$ for a smooth function $\varphi: \H\times (0, \infty)\to \mathbb{R}$, we have
\[
\begin{aligned}
&\varphi_t(\hat{p}, \hat{t}) +F_\ast(\nabla_H \varphi(\hat{p}, \hat{t}), \nabla_H^2 \varphi(\hat{p}, \hat{t})) \leq 0 \\
& \left(\text{resp., } \quad \varphi_t(\hat{p}, \hat{t}) +F^\ast(\nabla_H \varphi(\hat{p}, \hat{t}), \nabla_H^2 \varphi(\hat{p}, \hat{t})) \geq 0 \right).
\end{aligned}
\]
A function $u\in C(\H\times (0, \infty))$ is called a viscosity solution of \eqref{mcf} if it is both a viscosity subsolution and a viscosity supersolution of \eqref{mcf}.
\end{defi}

The definition above employs the semicontinuous envelopes to overcome the singularity of $F(\xi, X)$ at $\xi=0$, which essentially corresponds to the use of $L^\ast$ in our preceding study of h-quasiconvexity. Note that for any smooth $f: \H\to \R$ and $p\in \H$, we have 
\[
F(\nabla_H f(p), \nabla_H^2 f(p))=-L[f](p) 
\]
by \eqref{sec operator heis2} if $\nabla_H f(p)\neq 0$, while
\[
F_\ast(\nabla_H f(p), \nabla_H^2 f(p))=-L^\ast[f](p)
\]
holds even if $\nabla_H f(p)=0$. Such connection enables us to utilize $L^\ast$ to investigate h-quasiconvexity preserving property for the horizontal curvature flow in Section~\ref{sec:preserving}. 

Let us next review the game-theoretic approach to \eqref{mcf}. The game starts at a given point $p\in \H$ with a fixed duration $t\geq 0$. The step size is denoted by $\vep$ and the total number of steps is $[t/\vep^2]$. Two players play the game, following the repeated rules below.
\begin{itemize}
\item Player I chooses a direction $v\in \H_0$ with $|v|=1$.
\item Player II determines a value $b=\pm1$. 
\item Once the decisions are made, the game position move from the current position $p$ to $p\cdot \sqrt{2}\vep bv$. 
\end{itemize} 
We denote by $y_k$ the game position after $k$ steps. Player I and Player II are trying to minimize and maximize the value $u_0(y_N)$ respectively. The value function is defined to be
\[
u_\vep(p, t)=\min_{v_1\in \H_0, |v_1|=1}\max_{b_1=\pm 1}\ldots \min_{v_N\in \H_0, |v_N|=1}\max_{b_N=\pm 1} u_0(y_N).
\]
It is shown in \cite{FLM1} that $u_\vep$ converges locally uniformly to the unique solution $u$ of \eqref{mcf}\eqref{initial} in $\H\times [0, \infty)$ under the assumptions that $u_0$ is rotationally symmetric with respect to $z$-axis and takes constant value outside a compact set. These additional assumptions are actually used only to prove (CP). One can write a general result in the following way. 
\begin{thm}[Game-theoretic approximation]
Suppose (CP) holds. Let $u_\vep$ be the value function introduced above with a given $u_0\in C(\H)$. Assume that there exist $C>0$ and a compact set $K\subset \H$ such that $u_0=C$ in $\H\setminus K$. Then $u_\vep\to u$ locally uniformly as $\vep\to 0$, where $u$ is the unique viscosity solution of \eqref{mcf}\eqref{initial}. 
\end{thm}

The key ingredient of the game-theoretic approach is the so-called dynamic programming principle (DPP), which is expressed as follows: 
\begin{equation}\label{dpp}
u_\vep(p, t)=\min_{v\in \H_0, |v|=1} \max_{b=\pm 1} u_\vep(p\cdot \sqrt{2}\vep bv, t-\vep^2) \quad \text{for all $p\in \H$ and $t\geq \vep^2$.}
\end{equation}
One can apply Taylor expansion to obtain \eqref{mcf} formally. The rigorous proof, using the notion of viscosity solutions, can be conducted in the same style. We omit the details, since the argument is somewhat similar to the proof of Proposition \ref{prop uqc vis}; see also the proof of Proposition \ref{prop uni-qc}. (The verification of supersolutions is slightly different, as shown in \cite{FLM1}.)

In general, $u_\vep$ is certainly not a continuous function in $\H\times [0, \infty)$. However, we see that it is always continuous in space due to the explicit iteration formula \eqref{dpp}. 

\subsection{H-quasiconvexity preserving property}\label{sec:preserving}

Let us study the h-quasiconvexity of solution in space. We begin with the case when the initial value is uniformly h-quasiconvex in $\H$. 
\begin{prop}[Iteration of uniform h-quasiconvexity]\label{prop time-mono}
Assume that ${u}_0\in C(\H)$ is uniformly h-quasiconvex in the sense of \eqref{uni-qc} with $\Omega=\H$, $r_0>0$ and $\lambda>0$. Let ${u}_\vep$ be the game value with step size $0<\vep<r_0/\sqrt{2}$. Then $u_\vep$ satisfies 
\begin{equation}\label{eq:mono-game}
u_\vep(p, t)\geq u_\vep (p, s)+2\lambda \vep^2[(t-s)/\vep^2]
\end{equation}
for all $p\in \H$ and $t\geq s\geq 0$. 
\end{prop}
\begin{proof}
Adopting  \eqref{uni-qc} holds for $f=u_0$, we can obtain, for $\vep<r_0/\sqrt{2}$, to get
\begin{equation}\label{eq:mono2}
u_\vep(p, \vep^2)=\min_{v\in \H_0, |v|=1} \max_{b=\pm 1} u_0(p\cdot \sqrt{2}\vep v)\geq u_0(p)+2\lambda \vep^2
\end{equation}
for all $p\in \H$. By the monotonicity of the game value with respect to the terminal cost, 
we have, for all $p\in \H$, 
\[
\begin{aligned}
u_\vep(p, 2\vep^2) &= \min_{v\in \H_0, |v|=1} \max_{b=\pm 1} u_\vep(p\cdot \sqrt{2}\vep bv, \vep^2)\\
&\geq \min_{v\in \H_0, |v|=1} \max_{b=\pm 1} u_0(p\cdot \sqrt{2}\vep bv) +2\lambda \vep^2= u_\vep(p, \vep^2)+2\lambda \vep^2.
\end{aligned}
\]
We can continue repeating the argument to deduce
\[
u_\vep(p, t)\geq u_\vep (p, t-\vep^2)+2\lambda \vep^2
\]
for all $t\geq 0$ and $p\in \H$.
Then \eqref{eq:mono-game} follows immediately from further iterations. 
\end{proof}

%If ${u}_0\geq C$ in $\H\setminus K$ for a compact set $K\subset \H$ and $C>0$, then by Proposition \ref{prop time-mono}, we have $u_\vep(p, t)\geq C$ for all $p\in \H\setminus K$ and $t\geq 0$. 

Letting $\vep\to 0$, we can get the uniform h-quasiconvexity of solution in space. In order for the game value to be locally bounded uniformly in $\vep>0$, we need to additionally impose a growth condition at space infinity as in \eqref{eq growth}. One can actually weaken the condition for more general initial values. 

\begin{prop}[Uniform h-quasiconvexity preserving]\label{prop uni-qc}
Assume that $u_0\in C(\H)$ is uniformly h-quasiconvex in $\H$ with parameters $r_0>0$ and $\lambda>0$. Suppose that there exists $L>0$ such that $u_0(p)\leq L(|p|_G^4+1)$ holds for all  $p\in \H$. Let $u_\vep$ be the game value with $\vep>0$ small. For any $t>0$, let $U^t$ be the relaxed upper limit of $u_\vep(\cdot, t)$, namely, 
\[
U^t(p)=\limsups_{\vep\to 0} u_\vep(\cdot, t)(p):= \lim_{r\to 0} \sup\{u_\vep(q, t): q\in B_r(p), \vep<r\}.
\]
%(We need further assumptions on $u_0$ to ensure $U^t<\infty$, say, there exists $C_1, C_2>0$ such that $u_0(p)\leq C_1|p|_G^4+C_2$ for all $p\in \H$.)
Then for any fixed $t>0$, $U^t\in USC(\H)$ is locally bounded and satisfies $\ol{L(U^t)}\geq 2\lambda$ in the viscosity sense in $\H$. In particular, $U^t$ is h-quasiconvex in $\H$. 
\end{prop}
\begin{proof}
Under the growth condition \eqref{eq growth}, we can follow \cite[Lemma 5.3]{FLM1} to obtain the following  local boundedness of $u^\vep$ uniform in $\vep>0$: for any compact set $K\subset \H\times [0, \infty)$, there exists $C_K>0$ such that $\sup_{K}u_\vep\leq C_K$ for all $\vep>0$.

From the definition of the relaxed upper limit, it is then straightforward to have $U^t \in USC(\H)$. Moreover, it follows from \eqref{eq:mono-game} that for any compact set $K\subset \H\times [0, \infty)$, there exists $C_K'>0$ such that $\inf_K u_\vep \geq -C_K' $ for all $\vep>0$. Then it follows from our assumption that $U^t$ is locally bounded.

Furthermore, the result \eqref{eq:mono-game} in Proposition \ref{prop time-mono} implies that
\begin{equation}\label{eq:mono4}
u_\vep(p, t)+2\lambda \vep^2\leq \min_{v\in \H_0, |v|=1} \max_{b=\pm 1} u_\vep(p\cdot \sqrt{2}\vep bv, t).
\end{equation}
for all $(p, t)\in \H\times (0, \infty)$ and $\vep>0$ small. 

Fix $t>0$ arbitrarily. Suppose that there exist $p_0\in \H$ and smooth $\varphi$ and $U^t-\varphi$ attains a strict maximum at $p_0$. Then by the definition of $U^t$, there exists a sequence $p_\vep$, still indexed by $\vep>0$, such that as $\vep\to 0$, $p_\vep\to p_0$
and  
\begin{equation}\label{eq:uni-qc1}
u_\vep(p_\vep, t)-\varphi(p_\vep)\geq \sup_{B_r(p_0)} (u_\vep(\cdot, t)-\varphi)-\vep^3
\end{equation}
for some $r>0$. 

In view of \eqref{eq:mono4} and \eqref{eq:uni-qc1}, we have 
\[
\varphi(p_\vep)+2\lambda \vep^2\leq \min_{v\in \H_0, |v|=1} \max_{b=\pm 1}\varphi(p_\vep\cdot \sqrt{2}\vep bv)+\vep^3
\]
for all $\vep>0$ small. Write $v = (v_h,0)$ with $v_h \in \R^2$.
An application of the Taylor expansion yields
\[
-\min_{|v_h|=1} \left\{\sqrt{2}\vep|\la \nabla_H \varphi(p_\vep), v_h\ra |+{\vep^2} \la (\nabla_H^2 \varphi(p_\vep))^\star v_h, v_h\ra\right\}\leq -2\lambda \vep^2 +o(\vep^2).
\]
We thus can use the same proof of Proposition \ref{prop uqc vis} to obtain that 
\[
-L[\varphi](p_\vep)\leq -2\lambda +o(1), 
\]
for $\vep>0$ small. As a result, we obtain that
\[
-\ol{L[\varphi]}(p) = -\limsup_{q\to p} L[\varphi](q) \leq -\limsup_{\vep \to 0} L[\varphi](p_\vep) \leq -2\lambda.
\]
This proves that $\ol{L(U^t)}\geq 2\lambda$ in the viscosity sense in $\H$.
%Defining, for any $p\in \H$ and any smooth function $\varphi$, 
%\[
% \ol{L[\varphi]}(p)=\limsup_{q\to p} L(\varphi)(q),
%\]
%we obtain $- \ol{L[\varphi]}(p_0)\leq -2\lambda.$
%Although this form is weaker than the sufficient condition for h-quasiconvex functions shown in Antoni's notes, it is still sufficient to guarantee the h-quasiconvexity; see Remark \ref{rmk relaxed-L} below. 
Then the h-quasiconvex of $U^t$ in $\H$ follows from Theorem~\ref{sufficient1-heis}.  
\end{proof}

%{\color{blue}
%\begin{rmk}

 %on  more information of the initial data $u_0$ to get the uniform upper bound for $u_\vep$ on compacta.
%Here we provide one sufficient condition: there exist $C_1, C_2>0$ such that $u_0(p)\leq C_1|p|_G^4+C_2$ for all $p\in \H$. See \cite[Lemma 5.3]{FLM1} for more details. Our examples in the following always satisfy this condition.
%\end{rmk}
%}

\begin{comment}
\textcolor{blue}{
\begin{rmk}\label{rmk relaxed-L}
Note that $L(\phi)^\star$ is not the upper semicontinuous envelope of the operator $L$ but only the upper semicontinuous envelope of the function $L(\phi)$, for a given $\phi$, with respect to its variable. Then, we see that not only $L(u)>0$ but also $L(u)^\star>0$ (in the viscosity sense) is sufficient for h-quasiconvexity of $u\in USC(\Omega)$ in an h-convex set $\Omega\subset \H$. In fact, in Antoni's notes, it has been shown that, for the constructed test function
\[
\varphi(x, y, z)=u(w)+{K\over 2}\left((bx-ay)^2+\left(z-\frac{(ax+by)(bx-ya)}{2}\right)^2\right),
\] 
\[
\la \nabla_H \varphi, (a, b)\ra=0, \quad L(\nabla_H \varphi, (\nabla_H^2 \varphi)^\star)\leq 0.
\]
These calculations are not limited to the interior maximizer $g$ of $u-\varphi$, but can be applied to the points around $g$. Therefore we have 
\[
L(\varphi)^\star(g)=\limsup_{\xi\to g}L(\varphi)(\xi) \leq 0.
\]
Hence, one can replace $L(u)$ by its upper semicontinuous envelope $L(u)^\star$ to weaken the sufficient condition for h-quasiconvexity. 
\end{rmk}
}
\end{comment}

Our comparison principle (CP) holds only for bounded solutions taking constant value outside a compact set. In order to show Theorem \ref{cor uni-qc2}, we need to truncate the limit of the corresponding game values to obtain a unique solution that is h-quasiconvex in space and satisfies the required conditions in (CP). 

%available, we can further prove the following result. 

\begin{proof}[Proof of Theorem \ref{cor uni-qc2}]
Let $\hat{u}_\vep$ denote the game value corresponding to the terminal cost $\hat{u}_0$ and $\hat{U}^t$ denote its relaxed upper limit in the space variable. By \eqref{truncation}, it is not difficult to see from the game setting that 
\[
u_\vep(p, t)=\min\{\hat{u}_\vep(p, t), C\} \quad\text{for all $p\in \H$ and $t\geq 0$.}
\]
We have shown in Proposition \ref{prop uni-qc} that $\hat{U}^t$ is h-quasiconvex in $\H$ for all $t> 0$. Since $u_\vep$ converges locally uniformly to $u$, we have $u(\cdot, t)=\min\{\hat{U}^t, C\}$ in $\H$. This yields the h-quasiconvexity of $u(\cdot, t)$ in $\H$. Indeed, for any $p\in \H$ and $t>0$, by the h-quasiconvexity of $\hat{U}^t$,  we have
\[
\hat{U}^t(p)\leq \max\left\{\hat{U}^t(p\cdot h), \hat{U}^t(p\cdot h^{-1})\right\} 
\]
for any $h\in \H_0$. It follows that 
\[
\min\left\{\hat{U}^t(p), C\right\}\leq \max\left\{\min\left\{\hat{U}^t(p\cdot h), C\right\},\ \min\left\{\hat{U}^t(p\cdot h^{-1}), C\right\} \right\}, 
\]
which is equivalent to saying that 
\[
u(p, t)\leq \max\{u(p\cdot h, t),\ u(p\cdot h^{-1}, t)\}.
\]
This completes the proof.
\end{proof}

We next use approximation to consider general h-quasiconvexity preserving property. Assume that the initial value $u_0$ can be approximated by a sequence of functions $u_{0,j} \in C(\H)$, each of which satisfies the assumptions in Theorem \ref{cor uni-qc2}. 
%in the form of \eqref{truncation} uniformly. 

% h-quasiconvex function $\overline{f}_j\in C(\H)$. 
%\begin{enumerate}
%\item[(IA)] There exists a sequence $\overline{f}_j\in C(\H)$ uniformly h-quasiconvex such that
%\begin{equation}\label{ia1}
%f_j: =\min\{\overline{f}_j, C\}\leq u_0 \quad \text{in $\H$ for all $j\geq 1$ and,}
%\end{equation}
%\begin{equation}\label{ia3}
%f_j: =\min\{\overline{f}_j, C\}\to u_0 \quad \text{uniformly in $\H$ as $j\to \infty$.}
%\end{equation}
%\end{enumerate}
%In view of the definition of the ity, this requires that 
%there exist $\lambda_j>0$ and $r_j>0$ such that for any $0<r<r_j$ and any $p\in \H$, one has
%\[
%\min_{v\in \H_0, |v|=1}\max_{b=\pm 1} \hat{u}_{0,j}(p\cdot r b v)\geq \hat{u}_{0,j}%(p)+\lambda_jr^2. 
%\]
%Under such an assumption on $u_0$, we can obtain Theorem \ref{thm qc}. 

\begin{thm}[H-quasiconvexity preserving property with approximation]\label{thm qc}
Suppose that (CP) holds. Let $C\in \R$ and $K_0\subset \H$ be a compact set. Let $u_0\in C(\H)$ be an h-quasiconvex function and $u_0\equiv C$ in $\H\setminus K_0$. Assume that there exists a sequence  $\hat{u}_{0, j}\in C(\H)$ uniformly h-quasiconvex in $\H$ satisfying the assumptions on $\hat{u}_0$ in Theorem \ref{cor uni-qc2}. 
Let ${u}_{0, j} =\min\{\hat{u}_{0, j}, C\}$. Assume that ${u}_{0, j} = C$ outside $K_0$ for all $j=1, 2, \ldots$ and 
%and a sequence $u_{0, j}\in C(\H)$ such that
%\begin{equation}\label{ia1}
%f_j: =\min\{\overline{f}_j, C\}\leq u_0 \quad \text{in $\H$ for all $j\geq 1$ and,}
%\end{equation}
\begin{equation}\label{ia3}
{u}_{0, j}\to u_0 \quad \text{uniformly in $\H$ as $j\to \infty$.}
\end{equation}
Let $u$ be the unique solution of \eqref{mcf}\eqref{initial}. Then, $u(\cdot, t)$ is h-quasiconvex in $\H$ for all $t\geq 0$.
\end{thm}
\begin{proof}
For fixed $j \ge 1$, since $\hat{u}_{0, j}$ satisfies the assumptions on $\hat{u}_0$ in Theorem \ref{cor uni-qc2}, that is, $\hat{u}_{0, j}\in C(\H)$ is uniformly h-quasiconvex in $\H$ and 
\begin{equation}\label{eq growth approx}
\hat{u}_{0, j}(p)\leq L_j(|p|_G^4+1), \quad p\in \H 
\end{equation}
for some $L_j>0$. As a result, $u_{0,j} = \min\{\hat{u}_{0,j}, C\}$ satisfies the assumptions on $u_0$ in Theorem \ref{cor uni-qc2} and thus we see that the corresponding solution $u_j$ is h-quasiconvex in space for all $t\geq 0$ and $j\geq 1$.

\begin{comment}
Since ${u}_{0, j}$ satisfies the assumptions on $u_0$ in Theorem \ref{cor uni-qc2}, that is, there exist  uniformly h-quasiconvex functions $\hat{u}_{0,j} \in C(\H)$ such that $u_{0,j} = \min\{\hat{u}_{0,j}, C\}$ and 
\begin{equation}\label{eq growth approx}
\hat{u}_{0, j}(p)\leq L_j(|p|_G^4+1), \quad p\in \H 
\end{equation}
for some $L_j>0$, we see that the corresponding solution $u_j$ is h-quasiconvex in space for all $t\geq 0$ and $j\geq 1$. %The condition \eqref{ia1} and (CP) also imply that $u_j\leq u$ for all $j\geq 1$. 
\end{comment}

Also, as a limit of the associated game values, $u_j$ is nondecreasing in time, which, combined with  the condition that $u_{0, j}=C$ outside $K_0$, implies that $u_{j}(\cdot, t)=C$ outside $K_0$ for all $t\geq 0$ and $j\geq 1$. By \eqref{ia3} and the standard stability argument for \eqref{mcf}\eqref{initial} \cite[Theorem 6.1]{FLM1} under (CP), we see that $u_j\to u$ uniformly as $j\to \infty$. As an immediate consequence, we obtain the desired h-quasiconvexity of $u(\cdot, t)$ for all $t>0$. Indeed, for any fixed $p\in \H$ and $v\in \H_0$, the h-quasiconvexity of $u_j(\cdot, t)$ yields
\[
u_j(p, t)\leq \max\{u_j(p\cdot h, t), u_j(p\cdot h^{-1}, t)\}\quad \text{for all $j\geq 1$.}
\]
%which, due to the relation $u_j\leq u$, further implies
%\[
%u_j(p, t)\leq \max\{u(p\cdot h, t), u(p\cdot h^{-1}, t)\}.
%\]
By the convergence of $u_j$ to $u$, it follows immediately that
\[
u(p, t)\leq \max\{u(p\cdot h, t), u(p\cdot h^{-1}, t)\},
\]
which gives the h-quasiconvexity of $u(\cdot, t)$.
\end{proof}

\subsection{Construction of initial functions}\label{sec:initial}

In our convexity preserving results in the previous section, we impose several assumptions on the initial value $u_0$. 
Our goal is to understand the h-convexity preserving property for the curvature flow with a given initial set $E_0$. Note that the geometric evolution $E_t$ does not depend on the choice of $u_0$ as long as \eqref{initial zero} holds and $u_0=C$ outside a compact set of $\H$; this can be proved by applying \cite[Theorem A.1]{FLM1} together with (CP). 
However, we need to clarify that such uniformly h-quasiconvex $u_0$ as in Theorem \ref{cor uni-qc2} or in Theorem \ref{thm qc} does exist. This is a highly nontrivial question. We below provide an affirmative answer in a special case when the following additional star-shapedness condition on the initial open set $E_0$ holds. The star-shapedness for solutions to elliptic equations in Carnot groups is studied in the recent work \cite{DGS}.

Let us assume that $E_0\subset \H$ is a nonempty open bounded set and satisfies the following conditions: 

\begin{enumerate}
\item[(S1)] $\delta_{\mu} (\overline{E_0})\subset E_0$ for any $0<\mu<1$;
\item[(S2)] There exist $r_0>0$ and $\sigma>0$ such that for any $0<r<r_0$, $p\in \partial E_0$ and $v\in \H_0$ with $|v|=1$, we have
\begin{equation}\label{starshape initial}
\max\{U_0(p\cdot rv), U_0(p\cdot rv^{-1})\}\geq 1+\sigma r^2,
\end{equation}
where $U_0: \H\to [0, \infty)$ denotes a Minkowski-type functional associated to $E_0$ given by 
\begin{equation}\label{minkowski}
U_0(p):=\begin{cases}\sup\left\{\mu^{-2}:\ \mu>0 \text{ such that } \delta_{\mu} (p)\notin E_0\right\} &  \text{if $p\neq 0$,}\\
0 & \text{if $p=0$.}
\end{cases}
\end{equation}
\end{enumerate}
The condition (S1) is a strict star-shapedess condition on $E_0$, while (S2) can be regarded as a reinforced h-convexity with the star-shapedness.

For $E_0\subset \H$ satisfying (S1)(S2), we use the function $U_0$ to build a uniformly h-quasiconvex function $\hat{u}_0\in C(\H)$ satisfying
\begin{equation}\label{initial zero2}
E_0=\{p\in \H: \hat{u}_0(p)< 0\},
\end{equation}
and the growth condition \eqref{eq growth} as well as the coercivity condition:
\begin{equation}\label{coercive}
\min_{p\in B_R(0)}\hat{u}_0(p) \to \infty \quad\text{as $R\to \infty$.}   
\end{equation}
Once this step is completed, one can truncate $\hat{u}_0$ as in \eqref{truncation} with $C\in \R$ large to get $u_0$ that meets our need for the h-quasiconvexity result in Theorem \ref{cor uni-qc2}.

\begin{prop}[Uniformly h-quasiconvex defining function]\label{prop initial}
Let $E_0\subset \H$ be an open bounded set satisfying the conditions (S1)(S2). Then there exists a uniformly h-quasiconvex function $\hat{u}_0\in C(\H)$ such that \eqref{initial zero2}, \eqref{coercive} and \eqref{eq growth} hold. 
\end{prop}

Before proving this proposition, we discuss several basic properties of $U_0$ for a star-shaped $E_0$. 
\begin{lem}[Properties of Minkowski-type functional]\label{lem minkowski}
Let $E_0\subset \H$ be a nonempty open bounded set satisfying (S1) and $U_0$ be given by \eqref{minkowski}. Then the following properties hold. 
\begin{enumerate}
\item[(i)] We have $0 \in E_0$, and $p=0$ if and only if $U_0(p)=0$.
\item[(ii)] For any $p\in \H\setminus \{0\}$ and $\mu>0$, $\delta_\mu(p)\in \ol{E_0}$ holds if and only if $U_0(p)\leq \mu^{-2}$.
\item[(iii)] For any $p\in \H\setminus \{0\}$ and $\mu>0$, $\delta_\mu(p)\notin E_0$ holds if and only if $U_0(p)\geq \mu^{-2}$. 
\item[(iv)] For any $p\in \H\setminus \{0\}$, $\delta_\mu(p)\in \partial E_0$ if and only if $U_0(p)=\mu^{-2}$.  %Consequently, combining with (ii) and (iii), we have $\delta_\mu(p)\in E_0$ holds if and only if $U_0(p) < \mu^{-2}$ and $\delta_\mu(p)\notin \ol{E_0}$ holds if and only if $U_0(p) > \mu^{-2}$. \textcolor{red}{I understand we use these two consequences more often below, but are they just direct equivalents to (ii)(iii) respectively? Maybe we can delete them?}
\item[(v)] For any $p\in \H$, $U_0(\delta_s(p))=s^2 U_0(p)$ for all $s\geq 0$. 
\item[(vi)] The function $U_0$ is continuous in $\H$. 
\item[(vii)] For $r, R>0$ satisfying $B_r(0)\subset E_0\subset B_R(0)$, there holds 
\begin{equation}\label{minkowski bound}
{|p|_G^2\over R^2}\leq U_0(p)\leq {|p|_G^2\over r^2} \quad \text{for all $p\in \H$.}
\end{equation}
Here $B_r(0), B_R(0)$ are gauge balls at the center $0$.
\end{enumerate}
\end{lem}
\begin{proof}
(i) {Pick a point $p \in E_0$. It follows from (S1) that $\delta_\mu(p) \in E_0$ for any $0 < \mu < 1$ and as a result $0 \in \ol{E_0}$. Then from (S1) again that $0 = \delta_{\frac12}(0) \in E_0$.} By definition of $U_0$, $p=0$ yields $U_0(p)=0$. If $p\neq 0$, then due to the boundedness of $E_0$, there exists $\vep>0$ small such that $\delta_{\vep^{-1}}(p)\notin E_0$, which implies that $U_0(p)\geq \vep^2>0$.

(ii) To prove ``$\Rightarrow$'' for $p\in \H\setminus \{0\}$, we use (S1) to get $\delta_s(p)\in E_0$ for all $0\leq s<\mu$, which immediately implies $U_0(p)\leq \mu^{-2}$. For the proof of ``$\Leftarrow$'', note that for any $0<s<\mu$ so that $s^{-2}>\mu^{-2}$, we can apply the definition of $U_0$ to deduce that $\delta_s(p)\in E_0$. This shows that $\delta_\mu(p)\in \ol{E_0}$.

(iii) ``$\Rightarrow$'' follows directly from the definition of $U_0$. The reverse implication can be obtained from (ii). In fact, assuming by contradiction that $\delta_\mu(p)\in E_0$ holds, we have $\delta_{\mu+\vep}(p)\in E_0$ for $\vep>0$ small, since $E_0$ is an open set. It follows from (ii) that $U_0(p)\leq (\mu+\vep)^{-2}$, which contradicts the condition that $U_0(p)\geq \mu^{-2}$.

(iv) This is an immediate consequence of (ii) and (iii).

(v) The case when $p=0$ or $s=0$ is trivial. Let us consider the case $p\neq 0$ and $s>0$. Since $\delta_\mu(p)=\delta_{\mu/s} (\delta_s(p))$, it is clear that $\delta_\mu(p)\notin E_0$ if and only if $\delta_{\mu/s} (\delta_s(p))\notin E_0$. Then, by the definition of $U_0$, we have
\[
U_0(p)=\sup\{\mu^{-2}: \delta_\mu(p)\notin E_0\}=s^{-2} \sup\left\{\left({\mu/s}\right)^{-2}:  \delta_{\mu/s} (\delta_s(p))\notin E_0 \right\} =s^{-2} U_0(\delta_s(p)).
\]
This homogeneity result actually does not require (S1). 

(vi) Fixing $p\in \H\setminus \{0\}$ and setting $\mu_p:=U_0(p)^{-1/2}>0$, for any fixed small $\vep>0$, it follows from (ii), (iii) and (iv) that  $\delta_{\mu_p-\vep}(p)\in E_0$ and $\delta_{\mu_p+\vep}(p)\notin \ol{E_0}$. As a consequence, 
we obtain $\delta_{\mu_p-\vep}(q)\in E_0$ and $\delta_{\mu_p+\vep}(q)\notin \ol{E_0}$ when $q\in \H$ is sufficiently close to $p$. It then follows from (ii) and (iii) again that 
\[
(\mu_p+\vep)^{-2} < U_0(q) < (\mu_p-\vep)^{-2}.
\]
We thus get $U_0(q)\to \mu_p^{-2}= U_0(p)$ as $q\to p$. In the case $p=0$ and $U_0(p)=0$, for any fixed $\vep>0$, we have $\delta_{\vep^{-1}}(q)\in E_0$ holds if $q\in \H\setminus\{0\}$ is taken sufficiently close to $0$. This yields $U_0(q)\leq \vep^2$, which, due to the arbitrariness of $\vep$, further implies $U_0(q)\to U_0(0)=0$ as $q\to 0$.

(vii) It is clear that \eqref{minkowski bound} holds at $p=0$. Let us fix $p\neq 0$. Since there exists $r>0$ such that $B_r(0)\subset E_0$, by definition of $U_0$ as in \eqref{minkowski} we have 
\[
U_0(p)\leq \sup\left\{\mu^{-2}:\ \mu>0 \text{ such that } \delta_{\mu} (p)\notin B_r(0)\right\}.
\]
Noticing that $\delta_s(p)\in B_r(0)$ if and only if %for all 
$0\leq s<r/|p|_G$, we are led to $U_0(p)\leq |p|_G^2/r^2$. Using the condition $E_0\subset B_R(0)$, we can similarly show that 
\[
U_0(p)\geq \sup\left\{\mu^{-2}:\ \mu>0 \text{ such that } \delta_{\mu} (p)\notin B_R(0)\right\}\geq {|p|_G^2\over R^2},
\]
which completes the proof of (vii).
\end{proof}

We now turn to the proof of Proposition \ref{prop initial}.
\begin{proof}[Proof of Proposition \ref{prop initial}]
%By assumptions, we see that $U_0\geq 0$ in $\H$, and $U_0(p)=0$ holds if and only if $p=0$. For any $p\in \H\setminus \{0\}$, we have $\delta_{\mu}(p)\in \partial E_0$ if and only if $\mu=U_0(p)^{-1/2}$. 

We first show that $U_0$ is uniformly h-quasiconvex in $\H\setminus B_\rho(0)$ for any $\rho>0$. Fix arbitrarily $p\in \H\setminus \{0\}$ and let $\mu_p=U_0(p)^{-1/2}$ again. It is clear that $\hat{p}=\delta_{\mu_p}(p)$ satisfies $U_0(\hat{p})=1$ and thus  $\hat{p}\in \partial E_0$ by Lemma \ref{lem minkowski}(iv)(v). In view of (S2), we have 
\[
\max\{U_0(\hat{p}\cdot rv), U_0(\hat{p}\cdot rv^{-1})\}\geq U_0(\hat{p})+\sigma r^2
\]
for all $r\in (0, r_0)$ and $v\in \H_0$ with $|v|=1$.
Since $U_0$ is homogeneous of degree $2$ with respect to the group dilation, as shown in Lemma \ref{lem minkowski}(v), we obtain 
\[
\max\left\{U_0(\delta_{\mu_p^{-1}}(\hat{p}\cdot rv)),\ U_0(\delta_{\mu_p^{-1}}(\hat{p}\cdot rv^{-1}))\right\}\geq U_0(p)+\sigma r^2 U_0(p).
\]
This amounts to saying that
\[
\max\{U_0(p\cdot rv), U_0(p\cdot rv^{-1})\}\geq U_0(p)+\sigma r^2 U_0(p)
\]
for all $p\in \H\setminus \{0\}$, $v\in \H_0$ with $|v|=1$ and $0<r< r_0 U_0(p)^{1\over 2}$. In particular, $U_0$ is uniformly h-quasiconvex in $\H\setminus B_\rho(0)$ for every $\rho>0$. We fix $\rho>0$ small such that $U_0<1/2$ in $B_\rho(0)$.

We finally construct $\hat{u}_0$ based on $U_0$. For $c>0$, {take $\psi_c(p)=c(x^2+y^2+|z|)+1/2$ for $p=(x, y, z)$. It follows from Example \ref{ex linear} that $\psi_c$ is}  %is clear that $\psi_c$ is uniformly convex in $\R^3$ and thus 
uniformly h-quasiconvex in $\H$. It is also easily seen that $\psi_c> U_0$ in $B_\rho(0)$. By choosing $c>0$ small, we have $\psi_c<U_0$ on $\partial E_0$ {(and thus on $E_0^c$ by homogeneity)}. Letting 
\[
{\hat{u}_0(p) := \max\{U_0(p), \psi_c(p)\}-1,} %\begin{cases}
%\max\{U_0(p), \psi_c(p)\}-1 &\text{if $p\in E_0$,}\\
%U_0(p)-1 & \text{if $p\notin E_0$,}
%\end{cases}
\]
we can verify that $\hat{u}_0$ is uniformly h-quasiconvex in $\H$ and satisfies  \eqref{initial zero2}. The verification of \eqref{initial zero2} is quite straightforward. Concerning the uniform h-quasiconvexity in $\H$, we only need to take arbitrarily $p\in \H$, $v\in \H_0$ with $|v|=1$ and $r\in (0, r_0)$ with $r_0>0$ sufficiently small and discuss three different cases in terms of the location of $p$: $p\in B_\rho(0)$, $p\in B_\rho(0)^c$ with $\psi_c < U_0$, and $p\in B_\rho(0)^c$ with $\psi_c \ge U_0$. %and $p\in \H\setminus E_0$. 
We omit the details here. 

The properties \eqref{coercive} and \eqref{eq growth} can also be easily proved by using Lemma \ref{lem minkowski}(vii). 
\end{proof}

Let us discuss the conditions (S1)(S2) more specifically under the rotational symmetry about the $z$-axis. In this special case, by expressing the boundary of $E_0$ by $x^2+y^2=g(z)$ as in Example \ref{ex rot-symmetry} and Example \ref{ex linear}, we can obtain more explicit sufficient conditions on $g$ that implies (S1)(S2). 

\begin{prop}[Uniformly h-quasiconvex initial value with rotational symmetry]\label{prop starshape sym}
Let $a, b\in \R$ with {$a < 0 < b$} %$a<b$ 
and $g\in C^2([a, b])$ %for some $\vep>0$ 
such that $g>0$ in $(a, b)$ and $g(a)=g(b)= 0$.  Let $E_0\subset \H$ be an open bounded set symmetric about the $z$-axis such that
\begin{equation}\label{rot graph eq}
E_0=\{(x, y, z)\in \H: x^2+y^2<g(z),\ z\in (a, b)\}.
\end{equation}
%for a nonnegative function for some $a, b\in \R$ with $a<b$. Assume that $g(z)>0$ if and only if $z\in (a, b)$. 
\begin{enumerate}
\item[(1)] If $g$ satisfies 
\begin{equation}\label{graph starshape0}
g(z)<{1\over \mu} g(\mu z)\quad \text{for all $0<\mu<1$ and $z\in [a, b]$,}
\end{equation}
then $E_0$ satisfies (S1). {In particular, if $g$ is monotonically increasing on $[a,0]$ and monotonically decreasing on $[0,b]$, then \eqref{graph starshape0} holds and $E_0$ satisfies (S1).}
\item[(2)] 
%Assume that there exists $\sigma>0$ small such that %$g$ can be extended to a function in $C^2([a-\sigma, b+\sigma])$ satisfying
If $g$ satisfies 
\begin{equation}\label{unif graph cond}
1-{2g(z) g''(z)\over 16+g'(z)^2}> \sigma\quad\text{for all $z\in [a, b]$, with some $\sigma \in (0,1)$}
\end{equation}
%and
%\begin{equation}\label{graph star cond}
%g(z)-g'(z)
%\end{equation}
then $E_0$ satisfies the conditions (S2). In particular, if $g''\leq 0$,  then \eqref{unif graph cond} holds with every $\sigma \in (0,1)$ and $E_0$ satisfies (S2).
 \end{enumerate}
\end{prop}
\begin{proof}
(1) Let $f(p)=x^2+y^2-g(z)$ for $p\in \H$.  For any $0<\mu<1$ and $q\in \delta_\mu(\ol{E_0})$, we can write $q=\delta_\mu (p)$ for some $p=(x, y, z)$ such that $x^2+y^2\leq g(z)$ with $z\in [a, b]$. By \eqref{graph starshape0}, we thus have 
\[
f(q)=\mu^2 \left(x^2+y^2-{1\over \mu^2} g(\mu^2 z)\right)< \mu^2\left(x^2+y^2-g(z)\right) =\mu^2 f(p)\leq 0,
\]
which yields $q\in E_0$.

(2) We extend $g$ to a function in $C^2([a-\vep, b-\vep])$ for some small $\vep>0$. %small such that {\color{yellow}$g>0$ in $(a, b)$ and $g\leq 0$ in $[a-\vep, a)\cup (b, b+\vep]$. Furthermore, by choosing $\vep$ small enough, we have}  
{ By a direct computation as in Example~\ref{ex rot-symmetry}, we have $L[f]>2\sigma$ in
\[
\ol{E_0} = \{(x,y,z) \in \H : x^2 + y^2 \le g(z), z \in [a,b]\}.
\]
In fact, by \eqref{unif graph cond}, for any $p=(x, y, z)\in \ol{E_0}$ satisfying $g''(z) \geq 0$, we get
\[
L[f](p)= 2-{4(x^2 + y^2) g''(z)\over 16+g'(z)^2} \geq 2-{4g(z) g''(z)\over 16+g'(z)^2} > 2\sigma.
\]
If on the other hand $g''(z) < 0$, it is easily seen that $L[f](p) \geq 2 > 2 \sigma$. As a result, by continuity we have $L[f] \ge 2 \sigma' > 2 \sigma$ in a neighbourhood of $\ol{E_0}$. %called $\Omega \subset \R^2 \times (a-\vep, b+\vep)$, 

Since $f(p) = 0$ on $\partial E_0$, it follows from Proposition~\ref{prop c2} that there exists $r_0 > 0$ such that
\[
\max\{f(p\cdot rv), f(p\cdot rv^{-1})\}\geq \sigma r^2
\]
for all $0 < r < r_0$, $p \in \partial E_0$ and $v \in \H_0$  with $|v| = 1$.
}
\begin{comment}
Following the computations in \eqref{eq graph taylor}, for $r_0>0$ small, we  have 
\[
\max\{f(p\cdot rv), f(p\cdot rv^{-1})\}\geq r^2-\frac{2g(z)g''(z)}{16+g'(z)^2}r^2+o(r^2)
\]
for any $r\in (0, r_0)$, $p=(x, y, z)\in \partial E_0$ and $v\in \H_0$. Here we applied the extension $g\in C^2([a-\vep, b+\vep])$ so that the Taylor expansion of $g$ around any $z\in [a, b]$ can be adopted. By \eqref{unif graph cond}, this implies
\[
\max\{f(p\cdot rv), f(p\cdot rv^{-1})\}\geq \sigma r^2.
\]
\end{comment}
Suppose that $f(p\cdot rv)\geq \sigma r^2$. Then writing $v=(\eta_1, \eta_2, 0)$, we get
\begin{equation}\label{starshape sym eq1}
(x+r\eta_1)^2+(y+r\eta_2)^2\geq g\left(z+{1\over 2}r x \eta_2-{1\over 2}r y \eta_1\right)+\sigma r^2.
\end{equation}
Noticing that 
\[
M := \sup_{z\in [a-\vep, b+\vep]} |g(z)|+|z g'(z)| < +\infty,
\]
we have 
\begin{equation}\label{graph starshape}
\begin{aligned}
\left(1+s \right)  g\left({z\over 1+ s}\right)
-g(z)\leq \int_0^s g\left({z\over 1+\tau}\right) -{z\over 1+\tau}g'\left({z\over 1+\tau}\right)\, d\tau\leq 2M s
 \end{aligned}
\end{equation}
for all {$s > 0$ and $z\in [a-\vep, b+\vep]$.} % with $\vep>0$ small. 
Applying \eqref{graph starshape} to \eqref{starshape sym eq1} with $s = \sigma_M r^2 := \frac{\sigma r^2}{2M}$, we obtain 
\[
(x+r\eta_1)^2+(y+r\eta_2)^2\geq (1+\sigma_M^2 r^2) g\left(\frac{z+{1\over 2}r x \eta_2-{1\over 2}r y \eta_1}{1+\sigma_M r^2}\right).
\]
%where we set $\sigma_M:=\sigma/(2M)$. 
This gives $f(\delta_{\mu_\sigma}(p\cdot rv))\geq 0$ with $\mu_\sigma:=(1+\sigma_M r^2)^{-1/2}$. In other words, $\delta_{\mu_\sigma}(p\cdot rv)\notin E_0$ and thus $U_0(p\cdot rv)\geq 1+\sigma_M r^2$ by Lemma \ref{lem minkowski}(iii). In the case that $f(p\cdot rv^{-1})\geq \sigma r^2$, we can similarly obtain $U_0(p\cdot rv^{-1})\geq 1+\sigma_M r^2$. Hence, we have \eqref{starshape initial} with $\sigma=\sigma_M$ for all $p\in \partial E_0$, $v\in \H_0$ and $r\in (0, r_0)$ when $r_0>0$ is sufficiently small, which verifies (S2).
\end{proof}

Based on the result above, we can find a class of h-quasiconvex initial sets for which the associated solutions of the horizontal curvature flow stay h-quasiconvex for all times.

\begin{prop}[General h-quasiconvex initial value with rotational symmetry]\label{prop initial approx}
Suppose that $E_0\subset \H$ is a bounded open h-convex set. Assume that $E_0$ is rotationally symmetric with respect to the $z$-axis and is expressed as in \eqref{rot graph eq} with $g\in C([a, b])$ and $a < 0 < b$ satisfying $g>0$ in $(a, b)$ and $g(a)=g(b)=0$. Assume that $g$ satisfies \eqref{graph starshape0}. Assume also that $g$ can be uniformly approximated in $[a, b]$ by a sequence of functions $g_j\in C^2([a_j, b_j])$ with $a_j\leq a, b_j\geq b$, {$a_j \to a, b_j \to b$ as $j \to \infty$} satisfying the conditions in Proposition \ref{prop starshape sym}; more precisely, $g_j>0$ in $(a_j, b_j)$, $g_j(a_j)=g_j(b_j)=0$ and %and can be extended to $C^2([a_j-\vep_j, b_j+\vep_j])$ for $\vep_j>0$ such that $g_j\leq 0$ in $[a_j-\vep, a_j]\cup [b_j, b_j+\vep_j]$, 
\begin{equation}\label{initial approx cond1}
\mu g_j(z)< g_j(\mu z)\quad \text{for all $0<\mu<1$ and for all $z\in [a_j, b_j]$,}
\end{equation}
and there exists $\sigma_j>0$ satisfying
\begin{equation}\label{initial approx cond2}
1-{2g_j(z) g_j''(z)\over 16+g_j'(z)^2}> \sigma_j\quad\text{ for all $z\in [a_j, b_j]$.}
\end{equation}
%and 
%\[
%g_j(z)+s\geq (1+\sigma_j s)g_j\left({z\over 1+\sigma_j s}\right)
%\]
%for all $s>0$ small and $z\in [a_j-\vep_j, b_j+\vep_j]$.
Set
\[
E_{0,j}=\{(x, y, z)\in \H: x^2+y^2<g_j(z),\ z\in (a_j, b_j)\}.
\]
Let $\hat{u}_0$ and $\hat{u}_{0, j}\in C(\H)$ be the functions constructed as in Proposition \ref{prop initial} for $E_0$ and $E_{0,j}$ respectively. Then \eqref{initial zero2} and the conditions in Theorem \ref{thm qc} hold. %, and $\hat{u}_{0, j}\to \hat{u}_0$ locally uniformly in $\H$ as $j\to \infty$. 
\end{prop}
\begin{proof}
{For $z \in [a_j,a]$ (resp., $z \in [b,b_j]$), by \eqref{initial approx cond1}, we have 
\[
 g_j(z)< \frac{z}{a} g_j(a), \quad \left(resp., \ g_j(z)< \frac{z}{b} g_j(b)\right)
\]
which implies $g_j(z) \to 0$ uniformly. As a result, we have}
$E_{0,j}\to E_0$ as $j\to \infty$ in the Hausdorff distance (associated to the Euclidean metric). Combining with the fact that $E_0$ satisfies (S1), it is not difficult to see that there exist $\lambda_j>1$ such that 
\[
\delta_{\lambda_j^{-1}}(E_0)\subset E_{0,j}\subset \delta_{\lambda_j}(E_0)
\]
and $\lambda_j\to 1$ as $j\to \infty$. It follows that $\lambda_j^{-2}\leq U_{0,j}(p)\leq \lambda_j^2$ for any $p\in \partial E_0$, which yields $U_{0,j}\to U_0$ uniformly on $\partial E_0$ as $j\to \infty$. Thanks to the homogeneity of $U_{0,j}$ and $U_0$ as in Lemma \ref{lem minkowski}(v), we get $|U_{0,j}(p)-U_0(p)|=U_0(p)|U_{0,j}(\hat{p})-U_0(\hat{p})|$ for all $p\in \H\setminus\{0\}$, where $\hat{p}=\delta_{\mu_p}(p)$ and $\mu_p=U_0(p)^{-1/2}$. Hence, for any compact set $K$, we obtain
\[
\sup_{p\in K} |U_{0,j}(p)-U_0(p)| \leq \sup_{p\in K} U_0(p)\ \sup_{q\in \partial E_0} |U_{0,j}(q)-U_0(q)|\to 0,
\]
as $j\to \infty$. {Thanks to this locally uniform convergence, we can choose the same constants $\rho$ and $c$ in the construction in Proposition \ref{prop initial} for $\hat{u}_{0, j}$ with $j$ large enough as well as $\hat{u}_0$.} %Constructing $\hat{u}_{0, j}$ and $\hat{u}_0$ by applying the same modification near $p=0$ with $\psi_c$ as in the proof of Proposition \ref{prop initial} respectively for $U_{0,j}$ and $U_0$, 
Then it is easy to show that $\hat{u}_{0, j}\to \hat{u}_0$ uniformly in any compact set $K$ as $j\to \infty$. Applying the star-shapedness of $E_0$ again, we have $U_0(p)<1$ for $p\in E_0$ and thus $\hat{u}_0$ fulfills \eqref{initial zero2}.

The growth condition \eqref{eq growth approx} holds for each $j$ due to Lemma \ref{lem minkowski}(vii). In addition, since $E_{0, j}$ are uniformly bounded, it also follows from Lemma \ref{lem minkowski}(vii) that there exist a compact set $K_0\subset \H$ and $C>0$ large such that $U_{0, j}> C$ in $\H\setminus K_0$ for all $j$. Then, taking the truncation $u_{0, j}=\min\{\hat{u}_{0, j}, C\}$, we see that $u_{0, j}$ and $\hat{u}_{0, j}$ satisfy all of the assumptions in Theorem \ref{thm qc}.
\end{proof}

%Since $\hat{u}_{0, j}\to \hat{u}_0$ locally uniformly as $j\to \infty$, under these assumptions we can take $C>0$ to deduce \eqref{ia3} for the truncated initial value $u_0$ in \eqref{truncation}. 

Let us provide two examples of rotationally symmetric sets that satisfy the assumptions of Proposition~\ref{prop starshape sym} or Proposition~\ref{prop initial approx}.
\begin{example}\label{ex gauge-ball}
Let $E_0\subset \H$ be the unit gauge ball in $\H$, that is, 
\begin{equation}\label{rot sym initial}
E_0=\{(x, y, z)\in \H: x^2+y^2< g(z)\}
\end{equation}
with $g$ given by $g(z)=\sqrt{1-16z^2}$ for $z\in [-1/4, 1/4]$. The h-convexity preserving property of the set evolution by curvature in this case can be observed directly from an explicit solution of \eqref{mcf} in \cite[Section 1.4]{FLM1}. We can prove this result also by Theorem \ref{thm qc}. 

Note that $g\notin C^2([-1/4, 1/4])$ and therefore Proposition~\ref{prop starshape sym} does not apply. However, we can approximate $g$ uniformly by $g_j$ satisfying \eqref{initial approx cond1} and \eqref{initial approx cond2} in Proposition~\ref{prop initial approx}.  These $g_j$ can actually be taken as 
\[
g_j(z)=\begin{dcases}
g(z) &\text{if $|z|<m_j$,}\\
g(-m_j)+g'(-m_j)(z+m_j)+{1\over 2} g''(-m_j)(z+m_j)^2 & \text{if $a_j\leq z\leq -m_j$,}\\
g(m_j)+g'(m_j)(z-m_j)+{1\over 2} g''(m_j)(z-m_j)^2 &\text{if $m_j\leq z\leq b_j$,} 
\end{dcases}
\]
where $m_j={1\over 4}-{1\over j}\in (0, 1/4)$ with $j\geq 1$ sufficiently large. Here $a_j<-1/4, b_j>1/4$ satisfy $g_j(a_j)=g_j(b_j)=0$. {Notice that for each $j$, $g_j$ is is monotonically increasing on $[a_j,0]$, monotonically decreasing on $[0,b_j]$, and $g_j'' \leq 0$ since we have $g'' \leq 0$ on $(-1/4, 1/4)$. As a result, \eqref{initial approx cond1} and \eqref{initial approx cond2} hold. See Proposition \ref{prop starshape sym} for example.} Then, Proposition \ref{prop initial approx} enables us to find $\hat{u}_{0, j}$ and $\hat{u}_0$ for the h-quasiconvexity preserving result in Theorem \ref{thm qc}.
\end{example}
\begin{example}
An example of $g$ for which $E_0$ in the form \eqref{rot sym initial} is h-convex in $\H$ but not convex in $\R^3$ is 
\[
g(z)=(1-z^2)(1+2z^2), \quad -1\leq z\leq 1.
\]
Since $g\in C^2([-1, 1])$ and for $z\in [-1, 1]$, we get
\[
g'(z)=2z-8z^3, \quad g''(z)=2-24z^2, \quad {0 \leq g(z) \leq \max_{[-1,1]} g=g\left(\pm {1\over 2}\right)={9\over 8}},
\] 
which yields
\[
{1\over 2}(16+g'(z)^2)> 7 > 2 g''(z) g(z),
\]
In other words, \eqref{unif graph cond} holds with $a=-1$, $b=1$ and $\sigma=1/2$. It is easily seen that \eqref{graph starshape0} holds as well. Indeed, for all $z\in [-1, 1]$ and $0<\mu<1$, by direct computations we have
\[
\begin{aligned}
g(\mu z)-\mu g(z)&=1-\mu+\mu(\mu-1)z^2-(2\mu^4-2\mu)z^4\\
&=(1-\mu)(1-\mu z^2+2\mu(1+\mu+\mu^2)z^4) {\ge (1-\mu)^2 > 0.}
%&\geq (1-\mu)(1-\mu z^2+2\mu z^4)\geq (1-\mu)\left(1-{\mu\over 8}\right)>{7\over 8}(1-\mu).
\end{aligned}
\]
Thus, we get $g(z)<g(\mu z)/\mu$ for all $z\in [-1, 1]$ and $0<\mu<1$. Hence,  such a set $E_0$ satisfies the assumptions of Proposition \ref{prop starshape sym}, and we can construct a uniformly h-quasiconvex function $\hat{u}_0\in C(\H)$ associated to $E_0$ as in the proof of Proposition \ref{prop initial}. 
\end{example}

\end{document}